\documentclass[11pt, reqno]{amsart}
\usepackage{amsmath}
\usepackage{mathrsfs}
\usepackage{yhmath}
\usepackage{amsxtra}
\usepackage{amscd}
\usepackage{amsthm}
\usepackage{amsfonts}
\usepackage{amssymb}
\usepackage{anysize}
\usepackage{accents}
\usepackage{bbding}
\usepackage{bbm}
\usepackage[bookmarks=true]{hyperref}
\usepackage{enumerate}
\usepackage{eucal}
\usepackage{fancyhdr}
\usepackage{mathtools,graphics,latexsym}
\usepackage{pifont}
\usepackage{pmgraph}
\usepackage{stmaryrd}
\usepackage[usenames,dvipsnames]{color}
\usepackage{tikz-cd}
\usepackage{upgreek}

\marginsize{2.9cm}{2.9cm}{2.15cm}{2.15cm}

\input prepictex
\input pictex
\input postpictex

\SetSymbolFont{stmry}{bold}{U}{stmry}{m}{n}

\newtheorem{thm}{Theorem}[section]
\newtheorem{cor}[thm]{Corollary}
\newtheorem{lem}[thm]{Lemma}
\newtheorem{prop}[thm]{Proposition}

\theoremstyle{definition}

\newtheorem{ex}[thm]{Example}
\theoremstyle{remark}
\newtheorem{rem}[thm]{Remark}

\numberwithin{equation}{section}

\begin{document}

\newcommand{\thmref}[1]{Theorem~\ref{#1}}
\newcommand{\secref}[1]{Section~\ref{#1}}
\newcommand{\lemref}[1]{Lemma~\ref{#1}}
\newcommand{\propref}[1]{Proposition~\ref{#1}}
\newcommand{\corref}[1]{Corollary~\ref{#1}}
\newcommand{\conjref}[1]{Conjecture~\ref{#1}}
\newcommand{\remref}[1]{Remark~\ref{#1}}
\newcommand{\eqnref}[1]{(\ref{#1})}
\newcommand{\exref}[1]{Example~\ref{#1}}

\DeclarePairedDelimiterX\setc[2]{\{}{\}}{\,#1 \;\delimsize\vert\; #2\,}
\newcommand{\bn}[1]{\underline{#1\mkern-4mu}\mkern4mu }

\makeatletter
\newcommand{\sbullet}{
  \mbox{\fontfamily{lmr}\fontsize{.6\dimexpr(\f@size pt)}{0}\selectfont\textbullet}}
\DeclareRobustCommand{\mathbullet}{\accentset{\sbullet}}
\makeatother

\newcommand{\nc}{\newcommand}

\nc{\Z}{{\mathbb Z}}
\nc{\Zp}{\Z_+}
\nc{\Zn}{\Z_-}
\nc{\C}{{\mathbb C}}
\nc{\N}{{\mathbb N}}

\nc{\hf}{\frac{1}{2}}
\nc{\thf}{\frac{3}{2}}
\nc{\bi}{\bibitem}
\nc{\wt}{\widetilde}
\nc{\wh}{\widehat}
\nc{\ov}{\overline}
\nc{\un}{\underline}
\nc{\bd}{\boldsymbol}
\nc{\mr}{\mathring}

\nc{\al}{\alpha}
\nc{\be}{\beta}
\nc{\ga}{\gamma}
\nc{\de}{\delta}
\nc{\ep}{\epsilon}
\nc{\La}{\Lambda}
\nc{\la}{\lambda}
\nc{\si}{\sigma}

\nc{\mc}{\mathcal}
\nc{\mf}{\mathfrak}
\nc{\cA}{\mc{A}}
\nc{\cC}{\mc{C}}
\nc{\cD}{\mc{D}}
\nc{\cE}{\mc{E}}
\nc{\cG}{\mc{G}}
\nc{\cL}{{\mc L}}
\nc{\cP}{\mc{P}}
\nc{\cS}{\mc{S}}
\nc{\cT}{\mc{T}}
\nc{\fA}{{\mf A}}
\nc{\fB}{{\mf B}}
\nc{\fb}{{\mf b}}
\nc{\fg}{{\mf g}}
\nc{\fh}{{\mf h}}
\nc{\fS}{\mf{S}}
\nc{\cB}{\mc{B}}
\nc{\z}{{\bf z}}

\nc{\s}{{\bf s}}
\nc{\hs}{\hat{s}}
\nc{\bs}{\bar{s}}
\nc{\bz}{\bar{0}}
\nc{\bo}{\bar{1}}

\nc{\Xl}{{\bd X}_{\! \ell}}
\nc{\Vac}{V_{{\rm crit}}}
\nc{\mn}{{m|n}}
\nc{\pn}{{p+m|q+n}}
\nc{\pdn}{{(p,q|m,n)}}
\nc{\pdnp}{{(p',q'|m',n')}}
\nc{\gs}{\mg_\pdn}
\nc{\gsr}{\fg_\prn}

\nc{\gl}{{\mf{gl}}}
\nc{\gls}{\gl_\pn}

\nc{\og}{{\ov{\fg}}}
\nc{\oh}{{\ov{\fh}}}
\nc{\ob}{{\ov{\fb}}}
\nc{\ogs}{{\ov{\fg}_{\pdn}}}

\nc{\mg}{\mathbbm{g}}
\nc{\mh}{\mathbbm{h}}
\nc{\mb}{\mathbbm{b}}
\nc{\mL}{\mathbbm{L}}
\nc{\omL}{\ov{\mathbbm{L}}}
\nc{\omLi}{\ov{\mathbbm{L}}^\iota}
\nc{\mV}{\mathbbm{V}}
\nc{\oE}{\ov E}

\nc{\mhs}{\mh_\pdn}
\nc{\mbs}{\mb_\pdn}

\nc{\obs}{{\ov{\fb}_\pdn}}
\nc{\ohs}{{\ov{\fh}_\pdn}}

\nc{\rla}{\mr\la}
\nc{\rlas}{\rla^\pdn}

\nc{\cLz}{\cL^\z}
\nc{\cLzs}{\cL^\z_\pdn}

\nc{\ocL}{\ov\cL}
\nc{\ocLz}{\ov\cL^\z}
\nc{\ocLzs}{\ov\cL^\z_\pdn}
\nc{\ocLs}{\ov\cL_\pdn}

\nc{\ta}{\tilde{a}}
\nc{\qm}{{q|m}}
\nc{\rr}{{r|r}}
\nc{\orr}{{(0,r|r,0)}}
\nc{\pp}{{\prime\prime}}

\nc{\hmn}{{\fh}_\mn}
\nc{\bmn}{{\fb}_\mn}

\nc{\hAz}{\wh{\cA}^\z}
\nc{\hP}{\wh \cP}
\nc{\oP}{\mr \cP}
\nc{\zpz}{\blb z^{-1}, \pz^{-1} \brb}

\nc{\hLs}{\wh{L}_\pdn}
\nc{\Ls}{L_\pdn}
\nc{\Az}{\cA^\z}
\nc{\Am}{\cA_{m}^\z}
\nc{\cLs}{\cL_\pdn}

\nc{\Xis}{\Xi_\pdn}
\nc{\vv}{|0\rangle}

\nc{\fz}{{\mf z}}
\nc{\fzg}{\fz(\wh{\mg}_{\rm crit})}
\nc{\fgm}{t^{-1} \mg [t^{-1}]}

\nc{\T}{{\rm T}}
\nc{\pz}{\partial_z}
\nc{\Ber}{{\rm Ber}}
\nc{\Bers}{{\rm Ber}^{\bf s}}
\nc{\cdet}{{\rm cdet}}
\nc{\sing}{{\rm sing}}
\nc{\End}{{\rm End}}
\nc{\tr}{{\mf{tr}}}
\nc{\otr}{\ov{\mf{tr}}}

\nc{\Imn}{\Is^{+}_\mn}
\nc{\Ipq}{\Is^{-}_{p|q}}
\nc{\Is}{\mathbb{I}}
\nc{\Js}{{\mathbb{J}}}
\nc{\Ks}{{\mathbb{K}}}
\nc{\Jp}{{\mathbb{K}^\prime}}
\nc{\Eii}{E_{i,i}}
\nc{\oEii}{\ov E_{i,i}}

\nc{\OO}{\mc{O}}
\nc{\Os}{\OO_\pdn}
\nc{\hgls}{\wh{\gl}_\pn}
\nc{\sigqmm}{\sig_{(q',m')}}
\nc{\sig}{{\bd \si}}
\nc{\gam}{{\bd \ga}}
\nc{\mga}{\mr \ga}
\nc{\mmu}{\mr \mu}
\nc{\hL}{\wh{L}}
\nc{\hcLs}{\wh{\cL}_\pdn}
\nc{\hcL}{\wh{\cL}}
\nc{\Lmqz}{\cL_{m|q}^\z}
\nc{\hLmqz}{\cL_{m|q}^\z}
\nc{\Lmqrz}{\cL_\mqr^\z}
\nc{\hLmqrz}{\wh{\cL}_\mqr^\z}

\nc{\Lt}{L^\theta}
\nc{\wL}{\wt{\un L}}
\nc{\wV}{\wt{V}}
\nc{\uiot}{\un \iota}

\nc{\hla}{\wh \la}
\nc{\ovla}{\ov{\la}}
\nc{\ovlas}{\ovla^\pdn}
\nc{\ovlamq}{\ovla^\mq}
\nc{\wtlamq}{\wt{\la}^\mq}
\nc{\opq}{\mathbbm{1}_{p|q}}

\nc{\Pmn}{\mc{P}_{\mn}}

\nc{\llangle}{\big\langle}
\nc{\rrangle}{\big\rangle}
\nc{\dpr}{{\prime \prime}}
\nc{\mm}{\mbox{-}}
\nc{\blb}{(\!(}
\nc{\brb}{)\!)}

\nc{\lei}{<_{\Is}}
\nc{\leqi}{\le_{\Is}}
\nc{\lesi}{<_{(\Is, \si)}}
\nc{\lepi}{<_{(\Is, \pi)}}
\nc{\lesig}{<_{(\Is, \sig)}}
\nc{\leqsi}{\le_{(\Is, \si)}}
\nc{\jlesi}{<_{(\Ks, \si)}}
\nc{\jleqsi}{\le_{(\Is, \si)}}
\nc{\cfd}{\text{CFD}}
\nc{\bD}{\textnormal{{\bf D}}}

\nc{\disp}{\displaystyle}
\nc{\ns}{\hspace{-0.3mm}}
\nc{\nns}{\hspace{-2mm}}
\nc{\tns}{\kern-1.1pt}

\advance\headheight by 2pt

\title[The Gaudin model and the super Bethe ansatz]
{The Gaudin model and the super Bethe ansatz for unitarizable modules over the general linear Lie (super)algebra}

\author[B. Cao]{Bintao Cao}
\address{School of Mathematics and Statistics, Yunnan University, Kunming 650500, China}
\email{btcao@ynu.edu.cn}

\author[W. K. Cheong]{Wan Keng Cheong}
\address{Department of Mathematics, National Cheng Kung University, Tainan 701401, Taiwan}
\email{keng@ncku.edu.tw}

\author[N. Lam]{Ngau Lam}
\address{Department of Mathematics, National Cheng Kung University, Tainan 701401, Taiwan}
\email{nlam@ncku.edu.tw}

\begin{abstract}

Let $\Az_\fg$ be the Gaudin algebra for the general linear Lie (super)algebra $\fg$ with respect to a sequence $\z \in \C^\ell$ of pairwise distinct complex numbers, and let $\mL$ be any $\ell$-fold tensor product of (infinite-dimensional) unitarizable highest weight $\fg$-modules. Fix a singular weight $\xi$ of $\mL$. We show that the singular weight space $\mL^\sing_\xi$ is a cyclic $\Az_\fg$-module, and that the Gaudin algebra $(\Az_\fg)_{\mL^\sing_\xi}$ for $\mL^\sing_\xi$ is a Frobenius algebra. We also show that $(\Az_\fg)_{\mL^\sing_\xi}$ is diagonalizable with a simple spectrum for a generic $\z$. Furthermore, we establish a super version of the Bethe ansatz for $\Az_\fg$ on $\mL^\sing_\xi$.

\end{abstract}

\maketitle

\setcounter{tocdepth}{1}

\section{Introduction}

The Gaudin model associated to a simple Lie algebra over $\C$ can be constructed using the notion of the center of an affine vertex algebra at the critical level (see \cite{FFR, FFTL, Ry06}).
The construction can be extended to the general linear Lie (super)algebra (see \cite{HM, MR, MTV06, Ta}).
Also, the Gaudin model contains (higher) Gaudin Hamiltonians, including the quadratic ones originally introduced by Gaudin to describe a completely integrable quantum spin chain \cite{G76, G83}.

There has been substantial progress in the study of the Gaudin model for the general linear Lie (super)algebra or a simple Lie algebra on finite-dimensional modules (see \cite{ChL25, FFRyb, HM, MTV06, MTV09-1, MTV09-2, MVY, Ry20}).
An understanding of the model on infinite-dimensional modules would be interesting and useful; relevant results have been obtained in \cite{CCL, ChL24, ChL26-1, ChL26-2}.
This paper will continue the investigation into the case of infinite-dimensional modules.

Fix a positive integer $\ell$.
Let $\Az_\fg$ be the Gaudin algebra for the general linear Lie (super)algebra $\fg$ with respect to a sequence $\z \in \C^\ell$ of pairwise distinct complex numbers.
It is the central object in the Gaudin model for $\fg$.
Let $\mL=L_1 \otimes \cdots \otimes  L_\ell$, where each $L_i$ is a $\fg$-module, and let $\mL^\sing$ be the singular space of $\mL$.

In \cite{ChL25}, the action of $\Az_\fg$ on $\mL^\sing$, where each $L_i$ is an irreducible polynomial $\fg$-module, has been studied.
The main goal of this paper is to extend the results of \cite{ChL25} to the case where each $L_i$ is an (infinite-dimensional) unitarizable $\fg$-module.

For any $\Az_\fg$-module $M$, let $(\Az_\fg)_M$ denote the Gaudin algebra for $M$, which is defined to be the image of $\Az_\fg$ in the endomorphism algebra $\End(M)$ of $M$.
For any singular weight $\xi$ of $\mL$, we denote by $\mL^\sing_\xi$ the singular weight space of $\mL$ of weight $\xi$.
We have the following theorem.

\begin{thm} [\thmref{Main}] \label{main1}

Let $L_i$ be an (infinite-dimensional) unitarizable highest weight $\fg$-module for all $i=1, \ldots, \ell$, and let $\xi$ be a singular weight of $\mL$. Then

\begin{enumerate}[\normalfont(i)]

\item $\mL^\sing_\xi$ is a cyclic $\Az_\fg$-module.
\item  $(\Az_\fg)_{\mL^\sing_\xi}$ is a Frobenius algebra.
\item  $(\Az_\fg)_{\mL^\sing_\xi}$ is diagonalizable with a simple spectrum for a generic $\z$.

\end{enumerate}
\end{thm}

Our strategy for proving \thmref{main1} is to analyze the isomorphisms of vector spaces in \eqref{diagram}.
Using the notion of truncation functors in super duality and the properties of Berezinians, we will prove \thmref{main1} by reducing all of its assertions to the corresponding assertions about irreducible polynomial modules, which have been established in \cite{ChL25}.

The non-super case of \thmref{main1} merits special mention. 
In fact, it extends a famous result of \cite{MTV09-1, MTV09-2} (see also \cite{Ry20}) about the Gaudin algebra $\Az_{\gl_m}$ for the general linear Lie algebra $\gl_m$ on an $\ell$-fold tensor product of irreducible finite-dimensional $\gl_m$-modules to an $\ell$-fold tensor product of infinite-dimensional unitarizable highest weight $\gl_m$-modules.

The Bethe ansatz method (\cite{BF, FFR}) provides an explicit construction of common eigenvectors for $\Az_{\gl_m}$ on the singular weight space $V^\sing_\mu$ of weight $\mu$, where $V$ is an $\ell$-fold tensor product of irreducible finite-dimensional $\gl_m$-modules.
These eigenvectors are called \emph{Bethe vectors} and are labeled by the solutions of the Bethe ansatz equations (see \secref{BAM}).
There is a conjecture, called the \emph{completeness of the Bethe ansatz}, which asserts that the Bethe vectors form an eigenbasis for $(\Az_{\gl_m})_{V^\sing_\mu}$ for a generic $\z$.
The work of \cite{MV05} shows that the conjecture holds for several examples, whereas a counterexample is discovered in \cite{MV07}.

\begin{thm} [\thmref{gls-gls-eigen} and \thmref{gl-gls-eigen}] \label{main2}

Let $\fg$, $\mL$, and $\xi$ be as in \thmref{main1}.
There exist a general linear Lie (super)algebra $\cG$, an $\ell$-fold tensor product $\mV$ of some irreducible polynomial $\cG$-modules, and a singular weight $\mu$ of $\mV$ (which depend on $\fg$, $\mL$, and $\xi$) such that there is a one-to-one correspondence between the set of eigenbases for $(\Az_\cG)_{\mV^\sing_\mu}$ and the set of eigenbases for $(\Az_\fg)_{\mL^\sing_\xi}$.

\end{thm}

\thmref{main2} provides an effective method of constructing an eigenbasis for $(\Az_\fg)_{\mL^\sing_\xi}$ from any eigenbasis for $(\Az_\cG)_{\mV^\sing_\mu}$, where $\cG$ can be chosen to be non-super.
If the completeness of the Bethe ansatz holds for $(\Az_\cG)_{\mV^\sing_\mu}$, then we will obtain an explicit eigenbasis for $(\Az_\fg)_{\mL^\sing_\xi}$ for a generic $\z$, which is described in terms of the Bethe vectors in $\mV^\sing_\mu$ and a certain element (in the universal enveloping algebra of some general linear Lie superalgebra containing both $\fg$ and $\cG$ as subalgebras) determined by a sequence of odd reflections.
The details will be given in \thmref{super-Bethe}, which we call the \emph{super Bethe ansatz} for $\Az_\fg$ on $\mL^\sing_\xi$.

We organize the paper as follows.
In \secref{Pre}, we set up our notations for general Lie (super)algebras and introduce truncation functors and unitarizable modules.
In \secref{Gaudin-alg}, we review the construction of the Gaudin algebras and discuss their basic properties.
\secref{Gaudin-Uni} details the assertions of \thmref{main1} while \secref{pf} is devoted to proving them.
In \secref{Bethe-ans}, we prove \thmref{main2} and establish the super Bethe ansatz.

\vskip 3mm
\noindent{\bf Notations.}
Throughout the paper, the symbol $\Z$ (resp., $\Zp$, $\Zn$ and $\N$) stands for the set of all (resp., non-negative, non-positive and positive) integers, the symbol $\C$ for the field of complex numbers, and the symbol $\Z_2:=\{\bz, \bo\}$ for the field of integers modulo 2. All vector spaces, algebras, tensor products, etc., are over $\C$.
{\bf We fix $p, q, m, n \in \Zp$.}
\vskip 3mm

\section{Preliminaries} \label{Pre}

In this section, we review general linear Lie (super)algebras, central extensions and parabolic categories (particularly truncation functors), and unitarizable modules.

\subsection{The general linear Lie (super)algebra}  \label{gls}
\sloppy
For $p, q, m, n \in \Zp$ that are not all zero, let
$$
 \Is=\Ipq \cup \Imn ,
$$
where
$\Ipq:=\{ -p, \ldots, -2, -1\} \cup \left\{-q+\hf, \ldots, -\thf, -\hf \right\}$
and $\Imn:=\{ 1, 2, \ldots, m\} \cup \left\{\hf, \thf, \ldots, n-\hf \right\}$ are subsets of $\hf \Z$.
We consider the total order on $\Is$:
$$
-p \lei \cdots \lei -1 \lei -q+\hf \lei \cdots \lei -\hf \lei 1 \lei \cdots \lei m \lei \hf \lei \cdots \lei n-\hf.
$$
For any $i, j \in \Is$, we say that $i \leqi j$ if $i \lei j$ or $i=j$.

Let $\{e_i \, | \, i \in \Is \}$ be a basis for the superspace $\C^{p|q}\oplus\C^{m|n}$ such that $\{e_i \, | \, i \in \Ipq \}$ and
$\{e_{i} \, | \, i \in \Imn \}$ are respectively the standard homogeneous bases for $\C^{p|q}$ and $\C^{m|n}$. In other words, the parity of $e_i$ is given by
$|e_i|=\ov{2i}\in \Z_2$ for $i \in \Is$.
For any $i, j \in \Is$, let $E_{i,j}$ denote the $\C$-linear endomorphism on $\C^{p|q}\oplus\C^{m|n}$ defined by
$$
E_{i,j} (e_k)=\de_{j, k} e_i \quad \mbox{ for $k \in \Is$,}
$$
where $\de$ is the Kronecker delta. The parity of $E_{i,j}$ is given by $|E_{i,j}|=\ov{2(i+j)}$.

Let $\gs$ denote the superspace of $\C$-linear endomorphisms on $\C^{p|q}\oplus\C^{m|n}$ with a homogeneous basis $\{ E_{i,j} \, | \, i, j \in \Is \}$.
It has a natural structure of a Lie superalgebra with commutation relations given by
$$
[E_{i,j}, E_{k, l}]=\de_{j,k} E_{i,l}-(-1)^{4(i+j)(k+l)}\de_{i, l} E_{k,j}, \qquad \mbox{for $i, j, k, l \in \Is$.}
$$
The Lie superalgebra $\gs$ is isomorphic to the general linear Lie (super)algebra $\gls$ associated to the standard homogeneous basis for the superspace $\C^{p+m|q+n}$.
Let
$$
\mbs=\bigoplus_{\substack{ i,j \in \Is, i \le_\Is j} }  \C E_{i,j}
$$
be a Borel subalgebra of $\gs$.
The corresponding Cartan subalgebra $\mhs$ of $\gs$ has a basis $\{ \Eii \, | \, i \in \Is \}$.
The dual basis in $\mhs^{*}$ is denoted by the set $\{ \ep_i \, | \, i \in \Is \}$, where the parity of $\ep_i$ is given by $|\ep_i|=\ov{2i}$.
For $\xi \in \mh_\pdn^{*}$, let $L(\gs, \xi)$ denote the irreducible highest weight $\gs$-module with highest weight $\xi$ with respect to $\mbs$.

\subsection{Central extensions and parabolic categories} \label{P-cat}
\subsubsection{Central extensions}
Let $\ogs$ denote the central extension of $\gs$ by a one-dimensional center $\C K$ determined by the 2-cocycle
$$
\uptau(A, B)=\text{Str}([\mathfrak{J},A]B) \qquad \text{for} \quad A,B\in \gs
$$
(cf. \cite[p. 115]{CLW12} or \cite[p. 99]{CL03}), where $\mathfrak{J}:=-\sum_{r \in \Ipq} E_{r,r}$, and $\text{Str}$ denotes the supertrace on $\gs$ defined by $\text{Str}(E_{i,j})=(-1)^{2i} \de_{i,j}$ for $i,j \in \Is$.
That is, $\ogs=\gs \oplus \C K$ as superspaces, and $\ogs$ is a Lie superalgebra with commutation relations given by
$$
[A+aK, B+bK]_{\ogs}:=[A, B]+\uptau(A,B)K.
$$
for $A, B \in \gs$ and $a, b \in \C$.
Also, $\uptau$ is a coboundary, and we have a Lie superalgebra isomorphism
$
\iota: \gs\oplus \C K \longrightarrow \ogs
$
given by
\begin{equation}\label{iso-e}
  \iota(A)=A+\mbox{Str}(\mathfrak{J} A)K, \quad \mbox{for $A \in \gs$, \quad and \quad $\iota(K)=K$.}
\end{equation}
We write $\oE_{i,j}:=(E_{i,j}, 0) \in \ogs$ for all $i, j \in \Is$.
We will always assume that every $\ogs$-module $M$ is viewed as a $\gs$-module obtained by
pulling back the action on $M$ via the isomorphism $\iota$.
We write $M^\iota$ to emphasize that $M$ is viewed as a $\gs$-module.

Let $\obs=\mbs \oplus \C K$ be a Borel subalgebra of $\ogs$.
The corresponding Cartan subalgebra $\ohs:=\mhs \oplus \C K$ of $\ogs$ has a basis $\{ \oEii \, | \, i \in \Is \}\cup \{K\}$.
The dual basis in $\oh_\pdn^*$ is given by $\{ \ep_i  \, | \, i \in \Is \} \cup \{\Lambda_0 \}$,
where $\La_0 \in \oh_\pdn^*$ is defined by
 $$
 \La_0(K)=1 \qquad \textrm{ and }   \qquad \La_0(\oEii)=0 \qquad \mbox{for $\, i \in \Is$.}
 $$
For $\xi \in \oh_\pdn^*$, let $L(\ogs, \xi)$ denote the irreducible highest weight $\ogs$-module with highest weight $\xi$ with respect to $\obs$.

\subsubsection{Parabolic categories}

Let $\Phi^+$ denote the set of all positive roots of $\ogs$ with respect to $\obs$, i.e.,
$$
\Phi^+=\setc*{\ep_i-\ep_j }{i,j \in \Is,\, i \lei j }.
$$
Let
$$
Y^+=\setc*{\ep_i-\ep_j }{i,j \in \Is,\, i \lei j \leqi -\hf \ \  \mbox{or}\ \  1\leqi i \lei  j }  \subseteq \Phi^+,
$$
and let
$$\ov{\mf{l}}_\pdn=\ohs \oplus \left(\bigoplus\limits_{\pm \al \in Y^+ } (\ogs)_\al \right)$$
be the Levi subalgebra of $\ogs$ associated to $Y^+$, where $(\ogs)_\al$ denotes the root space of $\ogs$ corresponding to $\al$.
Note that $\ov{\mf{l}}_\pdn \cong  \gl_{p|q} \oplus  \gl_{m|n} \oplus \C K$.
Define $\ov{\mf{p}}_\pdn=\ov{\mf{l}}_\pdn+\obs$, which is a parabolic subalgebra of $\ogs$.

A partition is a finite sequence $\mu:=(\mu_1, \mu_2, \ldots )$ of non-negative integers in decreasing order, i.e., $\mu_i \ge \mu_{i+1}$ for all $i \in \N$.
We denote by $\mu^\prime:=(\mu_1^\prime, \mu_2^\prime, \ldots)$ the conjugate partition of $\mu$ and by $\ell(\mu)$ the length of $\mu$.
A partition $\mu$ is called an \emph{$(m|n)$-hook partition} if $\mu_{m+1} \le n$ (or equivalently, $\mu'_{n+1} \le m$).
We denote by $\Pmn$ the set of all $(m|n)$-hook partitions.

For $a \in \C$, let
$$
\cP_\pdn^a=\{ \ovla=(\la^+, \la^-; a)\,|\, \la^+\in \Pmn, \la^-\in \cP_{p|q}\}
$$
and
$$
 \cP_\pdn=\bigcup_{a\in \C}\cP_\pdn^a.
$$
Define $\langle r \rangle={\rm max}\{r, 0 \}$ for $r \in \Z$.
For $\ovla \in \cP_\pdn^a$, we define
{\allowdisplaybreaks
\begin{eqnarray}\label{hlas}
 \ovlas=& &-\sum_{r=1}^{p} \langle \la^{-}_{r}-q \rangle \ep_{-r}- \sum_{r=1}^{q} (\la^{-})^\prime_{r} \ep_{-r+\hf} \nonumber  \\
& &  +\sum_{i=1}^m \la^+_i \, \ep_i+\sum_{i=1}^n\langle(\la^+)'_i-m\rangle \,\ep_{i-\hf }+a \La_0 \,\, \, \in \oh_\pdn^*.
\end{eqnarray}
}

\begin{rem}
Let $p',q',m',n' \in \Zp$ be such that $p'\le p$, $q' \le q$, $m' \le m$, and $n' \le n$.
For $\ovla\in \cP_{(p',q|m,n)}^a$ (resp., $\cP_{(p,q|m,n')}^a$), it is obvious that $\ovla^{(p',q|m,n)}=\ovlas$ (resp., $\ovla^{(p,q|m,n')}=\ovlas$).
But for $\ovla\in \cP_{(p,q'|m,n)}^a$ (resp., $\cP_{(p,q|m',n)}^a$), $\ovla^{(p,q'|m,n)}\not=\ovlas$ (resp., $\ovla^{(p,q|m',n)}\not=\ovlas$) in general if $m'\not=m$ (resp., $q'\not=q$).
\end{rem}

For $\xi \in \oh_\pdn^*$, let $L(\ov{\mf{l}}_\pdn, \xi)$ be the irreducible highest weight $\ov{\mf{l}}_\pdn$-module with highest weight $\xi$ with respect to the Borel subalgebra $\ohs \oplus \big( \bigoplus_{ \al \in Y^+ } (\ogs)_\al \big)$.

Let $\Os$ be the category of $\ogs$-modules $M$ satisfying the conditions (\cite[Section II.D]{CCL}, \cite[Section 7]{CLW12}) (cf. \cite{CL10, CLW11}):
\begin{enumerate}[\normalfont(i)]

\item $M$ is a semisimple $\ohs$-module with finite-dimensional $\mu$-weight spaces $M_{\mu}$, where $\mu \in \oh_\pdn^*$.

\item $M$ decomposes over $\ov{\mf{l}}_\pdn$ into a direct sum of $L \big(\ov{\mf{l}}_\pdn, \ovla^\pdn \big)$ for $\ovla \in \cP_\pdn$.

\item There exist finitely many weights $(\ovla^1)^\pdn,\ldots,(\ovla^k)^\pdn$ for some $\ovla^1, \ldots \ovla^k \in \cP_\pdn$ (depending on $M$) such that if $\mu$ is a weight of $M$, then
$(\ovla^i)^\pdn-\mu$ is a linear combination of simple roots with coefficients in $\Zp$ for some $i$.

\end{enumerate}

The morphisms in the categories are even $\ogs$-homomorphisms, and the categories are abelian.
For $a \in \C$, let $\Os^a$ be the full subcategory of $\Os$ consisting of the objects $M \in \Os$ such that $Kv=av$ for all $v \in M$.
Evidently, $\Os =\bigoplus_{a \in \C} \Os^a$.

Define
$$
\Xis= \sum_{i \in \Ipq} \Zn \ep_{i}+ \sum_{i \in \Imn} \Z_+ \ep_i + \C \La_0.
$$
For $\varepsilon=0$ or $1$, we define
$$
{\Xis}(\ov\varepsilon)=
\setc[\Big]{\mu \in {\Xis} }{ \sum_{i \in \Is \cap (\hf+\Z)}\mu(\oE_{i,i})\equiv \varepsilon \,\,(\text{mod }2)}.
$$
For $M\in  \Os$, ${M}={M}_{\ov{0}}\bigoplus {M}_{\ov{1}}$  is a $\Z_2$-graded vector space, where
 \begin{equation}\label{wt-Z2-gradation}
{M}_{\ov\varepsilon}:=\bigoplus_{\mu \in \Xis(\ov\varepsilon)}{M}_{\mu}, \qquad \mbox{ for $\varepsilon= 0, 1.$}
 \end{equation}
The $\Z_2$-gradation on $M$ is clearly compatible with the action of $\ogs$.
Note that for $M, N \in \Os$, the $\Z_2$-gradation on $M\otimes N$ given by \eqnref{wt-Z2-gradation} coincides with the $\Z_2$-gradation on $M\otimes N$ induced from the $\Z_2$-gradations on $M$ and $N$ given by \eqref{wt-Z2-gradation}.

We have the following lemma (see the paragraph preceding \cite[Theorem 6.4]{CW}).
\begin{lem} \label{weight-Xi}
Let $M \in \Os$. If $\mu$ is a weight of $M$, then $\mu \in \Xis$.
\end{lem}

By \cite[Theorem 3.27 and Theorem 6.4]{CW} (see also the proof of \cite[Theorem 6.2.2]{Lus}), we have:

\begin{prop} \label{tensor-cat}
The category $\Os$ is a tensor category.
\end{prop}

The following lemma (cf. \cite[Lemma 4.1]{CLW12}) can be proved by an argument analogous to the proof of \cite[Lemma 3.1]{CLW11}.

\begin{lem} \label{verma-Os}
For any $\ovla \in \cP_\pdn^a$, we have $L(\ogs, \ovla^\pdn) \in \Os^a$.
\end{lem}

We also have the following.

\begin{lem}\label{weight=0}
 Let $M, N \in \Os$, and let $\mu$ and $\ga$ be weights of $M$ and $N$, respectively. Then
$$
(\mu+\ga)(\oEii)=0\quad\mbox{if and only if }\quad \mu(\oEii)=0 \quad \text{and} \quad \ga(\oEii)=0,\qquad
\mbox{for $i \in \Is$.}
$$
\end{lem}

\begin{proof}
This follows immediately from the property that $\mu(\oEii)\in \Zn$ and $\ga(\oEii)\in \Zn$, for $i \in \Ipq$, and $\mu(\oEii) \in \Zp$ and $\ga(\oEii) \in \Zp$, for $i \in \Imn$.
\end{proof}

\subsubsection{Truncation functors}\label{tr}
For $p', q', m', n' \in \Z_+$ with
$p' \le p$, $q' \le q$, $m' \le m$, and $n' \le n$,
we may regard $\Xi_\pdnp \subseteq \Xi_\pdn$.
The truncation functor
$\tr^\pdn_\pdnp: \Os \longrightarrow \OO_\pdnp$ is defined by
\begin{equation} \label{trunc-def}
\tr^\pdn_\pdnp(M)=\bigoplus_{\nu \in \Xi_\pdnp}M_{\nu}, \qquad \mbox{for $\, M\in \Os$}.
\end{equation}
For simplicity, we write $\tr_\pdnp=\tr^\pdn_\pdnp$.
Note that the direct sum on the right-hand side of \eqref{trunc-def} belongs to $\OO_\pdnp$
by \cite[Lemma 6.1]{CLW12} and the fact that
$[\oEii, A]=0$ for all $A \in \ov{\fg}_\pdnp$
and $i \notin \Is_\pdnp$.
Also, for $f \in {\rm Hom}_{\Os}(M, N)$,
the map
$\tr_\pdnp(f)$ is defined to be the restriction of
$f$ to $\tr_\pdnp(M)$.
Evidently, $\tr_\pdnp$ is an exact functor. By \lemref{weight=0}, $$\tr_\pdnp(M\otimes N)=\tr_\pdnp(M)\otimes \tr_\pdnp(N)$$
for $M, N \in \Os$, and hence 	$\tr_\pdnp$ is a tensor functor.

The following proposition is similar to \cite[Lemma 3.2]{CLW11} (cf. \cite[Proposition 6.9]{CW} and \cite[Proposition 7.5]{CLW15}) and can be proved by repeating the arguments therein. Note that $\cP_\pdnp^a\subseteq \cP_\pdn^a$ if  $p' \le p$, $q' \le q$, $m' \le m$, and $n' \le n$.

\begin{prop}\label{trun-sep}
For $p', q', m', n' \in \Z_+$ with $p' \le p$, $q' \le q$, $m' \le m$, and $n' \le n$ and $\ovla \in \cP_\pdn$,
Then
$$
 \tr_\pdnp \big(L \big(\ogs, \ovla^\pdn \big)\big)
=\begin{cases}
     L \big(\og_\pdnp, \ovla^\pdnp \big) &\,\,\, \mbox{if $\ovla\in \cP_\pdnp$};\\
       0                                                   &\,\,\,  \mbox{otherwise}.
     \end{cases}
$$
\end{prop}

\subsubsection{Singular spaces}
 {\bf From now on, we will fix $p, q, m, n \in \Z_+$ and write $\mg=\gs$, $\mb=\mbs$, $\mh=\mhs$, $\og=\ogs$, $\ob=\obs$  and $\oh=\ohs$.}

Let $\fS_\Is$ denote the symmetric group on $\Is$.
For $\si \in  \fS_\Is$, we endow a new total order $\lesi$ on $\Is$ by defining
\begin{equation} \label{order}
i \lesi j \qquad \mbox{if $\quad \si(i)<_\Is \si(j)$.}
\end{equation}
For any $i, j \in \Is$, we say that $i \leqsi j$ if $i \lesi j$ or $i=j$.
We denote by $\Is^\si$ the set $\Is$ with total order $\lesi$ .

Let $\mb^{\si}$ (resp.,$\ob^{\si}$) denote the Borel subalgebra of $\mg$ (resp., $\og$) with respect to $\lesi$. That is,
$$
\mb^{\si}=\bigoplus_{\substack{ i,j \in \Is, \, i \leqsi j} }  \C E_{i, j}\qquad
\text{and} \qquad \ob^{\si}=\Big( \bigoplus_{\substack{ i,j \in \Is, \, i \leqsi j}} \C \oE_{i, j} \Big) \oplus \C K.
$$
The Borel subalgebras $\mb^{\si}$ and $\mb$ (resp., $\ob^{\si}$ and $\ob$) share the same Cartan subalgebra $\mh$ (resp., $\oh$).

For any $\mg$(resp., $\og$)-module $M$, the \emph{$\si$-singular space} of $M$ (with respect to $\mb^\si$ (resp., $\ob^\si$)) is defined to be
$$
M^{\si\mm\sing}=\setc*{\! v \in M }{E_{i, j} v=0 \,\,  \mbox{(resp., $\ov E_{i,j} v=0$) \, for all\, $i, j \in \Is $ with $i\lesi j$}\!}.
$$
Any nonzero vector in $M^{\si\mm\sing}$ is called a $\si$-singular vector.
For any weight $\mu$ of $M$, we let $M^{\si\mm\sing}_\mu=M_\mu \cap M^{\si\mm\sing}$.
If $M^{\si\mm\sing}_\mu \not=0$, then $\mu$ is called a $\si$-singular weight of $M$, and $M^{\si\mm\sing}_\mu$ is called the $\si$-singular weight space of $M$ of ($\si$-singular) weight $\mu$.
If $\si=1$, we set $M^\sing=M^{\si\mm\sing}$ and $M^\sing_\mu=M^{\si\mm\sing}_\mu$ and replace any ``$\si$-singular" with ``singular".
For any $\og$-module $M$,
$$
(M^\iota)^{\si\mm\sing}=M^{\si\mm\sing}
$$
as vector spaces.

For $\xi \in \mh^*$ (resp., $\oh^*$), let $L^\si(\mg, \xi)$ (resp., $L^\si(\og, \xi)$) be the irreducible highest weight $\mg$(resp., $\og$)-module with highest weight $\xi$ with respect to $\mb^{\si}$ (resp., $\ob^\si$).
It is evident that $L^\si(\mg, \xi)=L(\mg, \xi)$ (resp., $L^\si(\og, \xi)=L(\og, \xi)$) if $\si=1$.

Let $\ov \fS_\Is$ denote the set of all nonidentity permutations $\si \in \fS_\Is$ such that $\Ipq$ and $\Imn$ are both invariant under $\si$ and such that the usual ordering of each of the subsets $\{ -p, \ldots, -1,  1, \ldots, m\}$ and $\left\{-q+\hf, \ldots, -\hf, \hf, \ldots, n-\hf \right\}$ is preserved in the totally ordered set $\Is^\si$.

The following proposition follows from the proof of \cite[Lemma 3.2]{CL10} (cf. \cite[Proposition 6.4]{CLW15}).

\begin{prop} \label{Lth}
Let $M\in \Os$ and $\si \in \ov \fS_\Is$. For any singular weight $\ov\xi$ of $M$,  there is an element $A \in U(\og)$, which is a product of elements in $\setc*{ \oE_{i,j} }{ i,j \in \Is, \, ij>0, \, i+j \notin \Z }$, such that $Av$ is a $\ob^\si$-singular vector of weight $\ov\xi^\si$ for all $v\in M_{\ov\xi}^{\sing}$. In particular, $L^\si(\og, \ov\xi^\si)=L(\og, \ov\xi)$.
Moreover, there is an isomorphism
$$
\varphi: M_{\ov\xi}^{\sing} \longrightarrow M_{\ov\xi^\si}^{\si\mm\sing},
$$
defined by $\varphi(v)=Av$ for $v \in M_{\ov\xi}^{\sing}$.
\end{prop}

\begin{rem}
The element $A \in U(\og)$ in the above proposition is determined by a sequence of odd reflections (see \cite[Section 3.1]{CL10} or \cite[Section 6.3]{CW}), which depends on $\ov\xi$ and $\si$.
\end{rem}

The permutations in $\ov\fS_\Is$ considered in the following example will be useful in \secref{pf}.

\begin{ex}\label{ex}

For $q', m' \in \Zp$ with $q' \le q$ and $m'  \le m$, let $\sig:=\sig_{(q',m')} \in \ov\fS_\Is$ be the permutation determined by the total order
\begin{eqnarray}\label{si-order}
&&-q+\hf \lesig \ldots \lesig -q'+\frac{3}{2}  \lesig -p \lesig \cdots \lesig -1   \nonumber\\
&&\lesig -q'+\hf \lesig \cdots \lesig -\hf  \lesig 1 \lesig \cdots \lesig m'   \nonumber \\
&&\lesig \hf \lesig \cdots \lesig n-\hf \lesig m'+1\lesig \cdots \lesig m.
\end{eqnarray}
For $\ovla \in \cP^a_\pdn$, let $\ov\xi=\ovlas  \in \oh_\pdn^*$ be defined in \eqref{hlas}.
We define $\ov\xi^{\sig}\in \oh_\pdn^*$ by
\begin{eqnarray}\label{si-weight}
\ov\xi^{\sig}:=& &- \sum_{r=q'+1}^{q} \left\langle (\la^{-})^\prime_{r} -p \right\rangle \ep_{-r+\hf} -\sum_{r=1}^{p} \langle \la^{-}_{r}-q' \rangle \ep_{-r} - \sum_{r=1}^{q'} (\la^{-})^\prime_{r} \ep_{-r+\hf}   \nonumber\\
& & +\sum_{i=1}^{m'}  \la^+_i  \ep_i+\sum_{i=1}^{n} \left\langle (\la^+)'_i -m' \right\rangle \ep_{i-\hf}+\sum_{i=m'+1}^{m} \langle \la^+_i-n \rangle \ep_i +a \La_0.
\end{eqnarray}
By \cite[Theorem 2.55]{CW}, $L^{\sig}(\og, \ov\xi^{\sig})=L(\og,\ov\xi)$.
If $\ovla\in \cP^a_{(p,q'|m',n)}$, then $\ov\xi^{\sig}=\ovla^{(p,q'|m',n)}$, and so
$$
L^{\sig} \big(\og, \ovla^{(p,q'|m',n)} \big)=L \big(\og, \ovlas \big).
$$
\end{ex}

\subsection{Unitarizable modules} \label{uni-mod}

Let us review $*$-superalgebras and their unitarizable modules. We refer the reader to \cite{CCL, ChL24, CLZ, LZ06} for related discussions.

A \emph{$*$-superalgebra} is a pair $(\cA, \omega)$, where $\cA$ is an associative superalgebra, and $\omega: \cA \longrightarrow \cA$ is an even anti-linear anti-involution, called a $*$-structure on $\cA$.
A homomorphism $f: (\cA,\omega) \longrightarrow (\cA^\prime, \omega^\prime)$ of $*$-superalgebras is a homomorphism of superalgebras satisfying $\omega^\prime \circ f = f \circ \omega$.
Let $(\cA,\omega)$ be a $*$-superalgebra, and let $V$ be an $\cA$-module.  A Hermitian form $\langle\cdot|\cdot\rangle$ on $V$ is said to be \emph{contravariant} with respect to $\omega$ if $\langle av_1 | v_2 \rangle=\langle v_1 |\omega(a)v_2 \rangle$ for all $a\in \cA$ and $v_1, v_2 \in V$.
An $\cA$-module equipped with a positive definite contravariant Hermitian form is called a \emph{unitarizable} $\cA$-module.

A \emph{Lie superalgebra with $*$-structure} is a pair $(\cG, \omega)$, where $\cG$ is a Lie superalgebra, and $\omega: \cG \longrightarrow \cG$ is an even anti-linear anti-involution, called a $*$-structure on $\cG$.
Note that $\omega$ is a $*$-structure on $\cG$ if and only if the natural extension of $\omega$ to $U(\cG)$ is a $*$-structure on $U(\cG)$.
A homomorphism $f: (\cG, \omega) \longrightarrow (\cG^\prime, \omega^\prime)$ of Lie superalgebras with $*$-structures is a homomorphism of Lie superalgebras satisfying $\omega^\prime \circ f = f \circ \omega$.
Let $(\cG, \omega)$ be a Lie superalgebra with $*$-structure, and let $V$ be a $\cG$-module. A Hermitian form $\langle\cdot|\cdot\rangle$ on $V$ is said to be \emph{contravariant} with respect to $\omega$ if
$\langle g v_1 |v_2 \rangle=\langle v_1 |\omega(g)v_1 \rangle$ for all $g \in \cG$ and $v_1, v_2 \in V$.
A $\cG$-module equipped with a positive definite contravariant Hermitian form is called a \emph{unitarizable} $\cG$-module.
It is clear that any unitarizable $\cG$-module is a unitarizable $U(\cG)$-module, and vice versa, and that any unitarizable $\cG$-module is completely reducible.

We will focus on unitarizable highest weight $\mg$-modules, which are the main objects of our study.
Before proceeding, we introduce some terminology.

The Lie superalgebra $\mg$ admits a $*$-structure $\omega: \mg \longrightarrow \mg$ (\cite[Section 3.2]{CLZ}), defined by
\begin{equation} \label{star}
 \sum_{i, j \in \Is} c_{ij} E_{i,j} \mapsto  \sum_{i, j \in \Is} (-1)^{[i]+[j]} \bar{c}_{ij} E_{j,i},
\end{equation}
where $\bar{c}_{ij}$ denotes the complex conjugate of $c_{ij} \in \C$, and
$$
[i]:=  \begin{cases}
      \, 0  & \,\,\mbox{if}  \ \  i \in \{ -p, \ldots, -1\},\\
      \, 1  & \,\, \mbox{otherwise}.
     \end{cases}
$$

Fix $d \in \N$. A \emph{generalized partition} $\la:=(\la_1,\ldots,\la_d)$ of depth $d$ is defined to be a sequence of integers in decreasing order, i.e., $\la_1 \ge \ldots \ge \la_d$.
We denote by $\oP^d$ the set of all generalized partitions of depth $d$.
Let
$$
\oP^d_\pdn=\setc*{\la \in \oP^d}{\mbox{$\la_{m+1}\le n$ and $\la_{d-p}\ge -q$}}.
$$
The condition $\la_{d-p}\ge -q$ is considered to be automatically satisfied whenever $d \le p$.

For $\la \in \oP^d$, we define
\begin{equation} \label{part-pm}
\la^+=\left(\langle \la_1 \rangle,\ldots,\langle \la_d \rangle \right)
\qquad \text{and} \qquad
\la^-=\left(\langle -\la_d \rangle,\ldots,\langle -\la_1 \rangle \right).
\end{equation}
Then $\la^+$ and $\la^-$ are partitions.
If $\la \in \oP^d_\pdn$, then $(\la^+)'_{n+1}\leq m$, $\la^{-}_{p+1}\leq q$ and $\ell(\la^+)+\ell(\la^-)\le d$.
Thus, $\ovla:=(\la^+, \la^-; d) \in \cP_\pdn^d$.

For $\la \in \oP^d_\pdn$, we define
$$
\mathbbm{1}_{p|q}=\sum_{r=1}^{p} \ep_{-r}-\sum_{r=1}^{q} \ep_{-r+\hf}
$$
and
{\allowdisplaybreaks
\begin{eqnarray}\label{rlas}
 \rlas=& &-\sum_{r=1}^{p} \langle \la^{-}_{r}-q \rangle \ep_{-r}- \sum_{r=1}^{q} (\la^{-})^\prime_{r} \ep_{-r+\hf} \nonumber  \\
& &  +\sum_{i=1}^m \la^+_i \, \ep_i+\sum_{i=1}^n\langle(\la^+)'_i-m\rangle \,\ep_{i-\hf }-d\mathbbm{1}_{p|q} \,\, \, \in \mh^*_\pdn.
\end{eqnarray}
}
For $d \in \Zp$, let \begin{equation}\label{un-weight}
{\mf X}^d_\pdn=\setc*{ \rlas \in \mh^*_\pdn}{ \la \in \oP^d_\pdn}.
\end{equation}
Note that if $d=0$, then $\la$ is declared to be empty and $\rlas:=0$.
In this case, $L(\mg, \rlas)$ is the one-dimensional trivial $\mg$-module with highest weight $0$.

Define
$$
 {\mf X}_\pdn=\bigcup_{d \in \Zp}{\mf X}^d_\pdn\qquad \text{and} \qquad \oP_\pdn=\bigcup_{d \in \Zp}\oP^d_\pdn.
$$
We have the following proposition (cf. \cite{CLZ}).

\begin{prop} \label{CLZ}
For any  $\la  \in \oP_\pdn$, $L(\mg, \rlas)$ is unitarizable with respect to $\omega$. Equivalently, $L(\mg, \xi)$ is unitarizable with respect to $\omega$ for all $\xi \in {\mf X}_\pdn$.
\end{prop}

The $*$-structure $\omega$ on $\mg$ extends to a $*$-structure $\omega: \mg\oplus \C K \longrightarrow \mg\oplus \C K$ by sending $K$ to itself.
On the other hand, the Lie superalgebra $\og$ possesses a $*$-structure $\ov\omega: \og \longrightarrow \og$ defined by $\ov\omega (\ov E_{i,j})= (\omega (E_{i,j}), 0)$ for all $i, j \in \Is$ and $\ov\omega(K)=K$.
Also,
\begin{equation}\label{om-io}
 (\ov\omega \circ \iota) (A)=(\iota \circ \omega) (A) \qquad \mbox{for all $A \in \mg\oplus \C K$}.
 \end{equation}
 Therefore, $\iota: (\mg\oplus \C K, \omega) \longrightarrow (\og, \ov\omega)$ is an isomorphism of Lie superalgebras with $*$-structures.

Fix $d \in \N$.
Let
$$
\ov{\mf X}^d_\pdn=\setc*{ \ovlas \in \oh_\pdn^*}{ \la \in \oP^d_\pdn}.
$$
For any $\xi \in {\mf X}^d_\pdn$, $L(\mg, \xi)$ is unitarizable and extends to a unitarizable highest weight $(\mg\oplus \C K)$-module by setting $K v= d v$ for all $v \in  L(\mg, \xi)$.
Also, $L(\og, \ov\xi)$ coincides with the $\og$-module obtained by pulling back the action on the extension of $L(\mg, \xi)$ via the isomorphism $\iota^{-1}$ and hence is unitarizable by \eqnref{om-io}.
Clearly, $L(\mg, \xi)=L(\og, \ov\xi)^\iota$.
Note that if $\xi=\rlas$ with $\la  \in \oP^d_\pdn$, then $\ov\xi=\ovlas \in \ov{\mf X}^d_\pdn$ by \eqnref{hlas}, \eqnref{part-pm}, and \eqnref{rlas}.
We summarize the discussion in the following proposition.

\begin{prop} \label{uni-K}
For any $\xi \in {\mf X}^d_\pdn$, there is $\ov\xi \in \ov{\mf X}^d_\pdn$ such that $L(\mg, \xi)=L(\og, \ov\xi)^\iota$.
Moreover, if $\xi=\rlas$ for some $\la  \in \oP^d_\pdn$, then $\ov\xi=\ovlas$.
\end{prop}

\section{Gaudin algebras} \label{Gaudin-alg}

The primary purpose of this section is to introduce the Gaudin algebra for the general linear Lie (super)algebra.
First, we introduce Berezinians and pseudo-differential operators.

\subsection{Berezinians}

Let $\cA$ be an associative unital superalgebra.
The parity of a homogeneous element $a \in \cA$ is denoted by $|a|$, which lies in $\Z_2$.

In this subsection, we let $r \in \N$, and let $\Ks$ be a set of size $r$ with total order $<_\Ks$.
Let $A=\big[a_{i,j}\big]_{i,j \in \Ks}$ be an $r \times r$ matrix. For any nonempty totally ordered subset $\Jp$ of $\Ks$, the matrix $A_{\Jp}:=\big[a_{i,j}\big]_{i,j \in \Jp}$ is called a \emph{standard submatrix} of $A$.

Assume that $A=\big[a_{i,j}\big]_{i,j \in \Ks}$ has a two-sided inverse $A^{-1}:=\big[\wt{a}_{i,j}\big]_{i,j \in \Ks}$.
For any $i, j \in \Ks$, the \emph{$(i,j)$th quasideterminant} of $A$ is defined to be $|A|_{i,j}:=\wt{a}_{j,i}^{-1}$ if $\wt{a}_{j,i}$ has an inverse in $\cA$.
Write $\Ks=\{k_1<_\Ks \cdots <_\Ks k_r\}$.
Following the notation of \cite{GGRW}, we write
$$
|A|_{k_i, k_j}=\begin{vmatrix}
 a_{k_1,k_1} 	& \ldots & a_{k_1,k_j} &\ldots & a_{k_1,k_r} \\
 \ldots	& \ldots & \ldots  &\ldots & \ldots	\\
 a_{k_i,k_1} & \ldots & \fbox{$ a_{k_i,k_j}$} &\ldots & a_{k_i,k_r}\\
 \ldots	& \ldots & \ldots  &\ldots & \ldots	\\
  a_{k_r,k_1}&\ldots & a_{k_r,k_j} &\ldots & a_{k_r,k_r}
\end{vmatrix}.
$$
For $i=1,\ldots,r$, we define
$$
d_{k_i}(A)=\begin{vmatrix}
	 a_{k_1,k_1} 	& \ldots & a_{k_1,k_i}\\
	 \ldots	& \ldots & \ldots \\
	 a_{k_i,k_1} & \ldots & \fbox{$a_{k_i,k_i}$}
\end{vmatrix},\
$$
called the \emph{principal quasiminors} of $A$.
We say that $A$ is \emph{sufficiently invertible} if every principal quasiminor of $A$ is well defined, and that $A$ is \emph{amply invertible} if every standard submatrix of $A$ is sufficiently invertible.

Let $m, n \in \Zp$ such that $m+n=r$.
Let $\cS_\mn$ be the set of all sequences $\s=(s_i)_{i \in \Ks}$ of 0's and 1's such that exactly $m$ of the $s_i$'s are 0 and the others are 1.
Every sequence in $\cS_\mn$ can be written in the form
$(0^{m_1}, 1^{n_1}, \ldots, 0^{m_k}, 1^{n_k}),$
where the sequence begins with $m_1$ copies of $0$'s, followed by $n_1$ copies of $1$'s, and so on.

Fix $\s =(s_i)_{i \in \Ks} \in \cS_\mn$. 
We say that $A=[a_{i,j}]_{i,j \in \Ks}$ is of type $\s$ if $a_{i,j}$ is a homogeneous element of parity $|a_{i,j}|=\bs_i+\bs_j$ for any $i,j \in \Ks$.
Suppose $A$ is a sufficiently invertible matrix of type $\s$ over $\cA$.
Following \cite{HM}, the \emph{Berezinian of type $\s$} of $A$ is defined to be
$$
\Bers (A)=d_{k_1} (A)^{\hs_{k_1} }\ldots d_{k_r}(A)^{\hs_{k_r}},
$$
where $\hs_i:=(-1)^{s_i}$ for $i \in \Ks$.
The original formulation of Berezinians, in the case where $\cA$ is supercommutative, was given in \cite{Ber} (see also \cite{MR}).

A matrix $A=[a_{i,j}]_{i,j \in \Ks}$ over $\cA$ is called a \emph{Manin matrix} of type $\s$ if $A$ is a matrix of type $\s$ satisfying the following relations
$$
[a_{i,j}, a_{k,l}]=(-1)^{s_i s_j+s_i s_k+s_j s_k}[a_{k,j}, a_{i, l}]  \qquad  \text{for all}  \quad i,j,k,l \in \Ks,
$$
where $[a, b]:=ab-(-1)^{|a||b|}ba$ for any homogeneous elements $a, b \in \cA$.

Let $\fS_{\Ks}$ be the symmetric group on $\Ks$.
Let $\si \in  \fS_\Ks$.
Analogous to \eqref{order}, we endow a total order $\jlesi$ on $\Ks$ by defining
$$
i \jlesi j \qquad \mbox{if $\quad \si(i)<_\Ks \si(j)$.}
$$
We write $\Ks^\si$ for the set $\Ks$ with total order $\jlesi$.
For any $\s =(s_i)_{i \in \Ks} \in \cS_\mn$ and any matrix $A=[a_{i,j}]_{i,j \in \Ks}$ over $\cA$, we define
$$
\s^\si=(s_i)_{i \in \Ks^\si} \qquad \text{and} \qquad A^\si=[a_{i,j}]_{i,j \in \Ks^\si}.
$$
Evidently, $\s^\si=\left(s_{\si^{-1}(i)} \right)_{i \in \Ks}$ and $A^\si=\big[a_{\si^{-1}(i),\si^{-1}(j)} \big]_{i,j \in \Ks}$.

We recall some basic facts about Manin matrices.

\begin{prop}[{cf. \cite[Section 3]{HM}}] \label{basics}

Let $A=[a_{i,j}]_{i,j \in \Ks}$ be a Manin matrix of type $\s=(s_i)_{i \in \Ks}$ over $\cA$. Then

\begin{enumerate}[\normalfont(i)]

\item If $\Jp$ is a nonempty totally ordered subset of $\Ks$, then $A_{\Jp}$ is a Manin matrix of type $(s_i)_{i \in \Jp}$.

\item For any $\si \in \fS_{\Ks}$, $A^\si$ is a Manin matrix of type $\s^\si$.

\end{enumerate}
\end{prop}

\begin{prop} [{\cite[Proposition 3.6]{HM}}] \label{HM}

Let $A=[a_{i,j}]_{i,j \in \Ks}$ be an amply invertible Manin matrix of type $\s$ over $\cA$.
Then
$$
 \hspace{1cm}  {\rm Ber}^{\s^\si} (A^\si)=\Bers (A)  \qquad \mbox{for any $\si \in \fS_{\Ks}$.}
$$
\end{prop}

Let $A=[a_{i,j}]_{i,j \in \Ks}$ be an $r \times r$ matrix over $\cA$. The \emph{column determinant} of $A$ is defined to be
$$
\cdet \,  A=\sum_{\si \in \fS_{\Ks}} (-1)^{l(\si)} \,  a_{\si(k_1),k_1}\ldots a_{\si(k_r),k_r},
$$
where $l(\si)$ denotes the length of $\si$.

\begin{prop}[{\cite[Lemma 8]{CFR}}] \label{CFR}
Let $\s_0=(0^r)$. For any sufficiently invertible Manin matrix $A$ of type $\s_0$ over $\cA$, we have $\Ber^{\s_0} (A)=\cdet (A)$.
\end{prop}

\subsection{Pseudo-differential operators} \label{diff-op}

Let $\cA$ be an associative unital superalgebra and $z$ an even variable.
We denote by $\cA \blb z \brb$ the superspace of formal Laurent series in $z$ with coefficients in $\cA$ and by $\cA\zpz$ the superspace of all formal series of the form
$$
  \sum_{j=-\infty}^s \sum_{i=-\infty}^r a_{i j} z^i \pz^j, \qquad  r, s \in \Z \quad \text{and} \quad a_{i j} \in \cA.
$$
The superspace $\cA \blb z \brb$ is a superalgebra, and we can give $\cA \zpz$ with a superalgebra structure using the rules:
\begin{equation} \label{rule}
\pz \pz^{-1}=\pz^{-1} \pz=1, \quad\pz^{i} z^j=\sum_{k=0}^{\infty}  \binom{i}{k} \binom{j}{k} k! \,   z^{j-k} \, \pz^{i-k},  \quad \mbox{for $i, j \in \Z$.}
\end{equation}
Here, e.g., $\disp{\binom{i}{k}:=\frac{i(i-1)\ldots (i-k+1)}{k!}}$.
We refer to \cite[Section 2.3]{ChL26-1} for further details.
The elements of $\cA\zpz$ are called pseudo-differential operators.

For any nonempty set $S$, we also let $S \zpz$ be the set of all formal series of the form
$$
  \sum_{j=-\infty}^s \sum_{i=-\infty}^r c_{i j} z^i \pz^j, \qquad  r, s \in \Z \quad \text{and} \quad c_{i j} \in S.
$$
For any $\cA$-module $M$, $v \in M$, and $\disp{f:=\sum_{j=-\infty}^s \sum_{i=-\infty}^r a_{ij} z^i \pz^j \in \cA\zpz}$, where $r, s \in \Z$ and $a_{i j} \in \cA$, we define
$$
f v=\sum_{j=-\infty}^s \sum_{i=-\infty}^r (a_{ij} v) z^i \pz^j \in M \zpz.
$$

\subsection{The Gaudin algebra for $\gs$}

For any Lie superalgebra $\fg$, we denote by $U(\fg)$ the universal enveloping algebra of $\fg$.
For any even variable $t$, let $\fg[t, t^{-1}]:=\fg\otimes \C[t, t^{-1}]$ denote the loop algebra of $\fg$, which is a Lie superalgebra with commutation relations
$$
\hspace{1cm}  \big[A \otimes t^r, B \otimes t^s \big]=[A, B] \otimes t^{r+s} \qquad \mbox{for $A, B \in \fg \,$ and $\, r, s \in \Z$.}
$$
Here $[A, B]$ denotes the supercommutator of $A$ and $B$.
We will use the notation
$$
\hspace{1cm} A[r]:=A \otimes t^r \qquad \mbox{for $A \in \fg\,$ and $\,r \in \Z$.}
$$
The superalgebras $\fg[t]:=\fg \otimes \C[t]$ and $t^{-1} \fg[t^{-1}]:=\fg \otimes t^{-1}\C[t^{-1}]$ are subalgebras of $\fg[t, t^{-1}]$ defined in an obvious way.
We identify $\fg$ with the subalgebra $\fg\otimes 1$ of constant polynomials in $\fg[t]$ and hence $U(\fg) \subseteq U(\fg[t])$.

Recall that $\mg:=\gs$.
Let $\Vac(\mg)$ denote the \emph{universal affine vertex algebra at the critical level} associated to the central extension $\wh{\mg}_{\rm crit}$ of the loop algebra $\mg[t,t^{-1}]$; see \cite{Fr07, FBZ, MR} for details. It is also a $\wh{\mg}_{\rm crit}$-module.

The \emph{Feigin--Frenkel center} $\fzg$ is defined to be the center of the vertex algebra $\Vac(\mg)$.
By the state-field correspondence (see \emph{op. cit.} and also \cite[Section 3.1]{ChL25}), $\fzg$ is equal to the subspace of $\mg[t]$-invariants on $\Vac(\mg)$.
That is,
$$
\fzg=\setc*{\! v \in \Vac(\mg)}{\mg[t] v=0 \!}.
$$
By the Poincar\'e--Birkhoff--Witt theorem, $\Vac(\mg)$ can be identified with $U(\fgm)$ (as superspaces).
Moreover, the superalgebra $U(\fgm)$ is equipped with the (even) derivation $\T$ defined by
$$
\T(1)=0 \qquad \text{and}\qquad \T(A[-r])=r A[-r-1]
$$
for $A \in \mg$ and $r \in \N$.
Note that $\T$ corresponds to the translator operator on $\Vac(\mg)$, and that $\fzg$ can be viewed as a commutative subalgebra of $U(\fgm)$ and is $\T$-invariant.

Fix $\s:=(0^p, 1^q, 0^m, 1^n) \in \cS_\mn$.
Recall the totally ordered set $\Is$ defined in \secref{gls}.
Let $\tau=-\partial_t$, where $t$ and $\partial_t$ satisfy the rules similar to \eqref{rule},
and let
$$
\cT_\pdn=\Big[\de_{i,j}\tau+(-1)^{2i}E_{i,j}[-1]\Big]_{i,j \in \Is}.
$$
It is known that $\cT_\pdn$ is an amply invertible Manin matrix of type $\s$ over $U(\fgm) \blb \tau^{-1} \brb$ (cf. \cite[Lemma 3.1]{MR}).
Furthermore,
\begin{equation} \label{Bers-T}
\Bers  \! \left(\cT_\pdn \right)=\sum_{i=-\infty}^{p+m-q-n}b_i \tau^i,
\end{equation}
for some $b_i \in U(\fgm)$.
Define $\wh{\fz}_\pdn$ to be the subalgebra of $U(\fgm)$ generated by
$$
\{\T^r b_i \, | \, i \le p+m-q-n, \, i \in \Z, \, r \in \Zp \}.
$$
By \cite[Corollary 3.3]{MR} (see also \cite[Proposition 3.3]{ChL25}),
$
b_i \in \fzg
$
for all $i \in \Z$ with $i \leq p+m-q-n$.
Since $\fzg$ is $\T$-invariant, we see that $\wh{\fz}_\pdn$ is a commutative subalgebra of $\fzg$.
In the non-super case, i.e., $q=n=0$, it is well known that $\fzg=\wh{\fz}_\pdn$ (see, e.g., \cite{FF, CM, Mo18}).
This is also true if $\mg=\gl_{1|1}$ (\cite{MM15}) or $\mg=\gl_{2|1}$ (\cite{AN}).
However, it is still unknown whether $\fzg=\wh{\fz}_\pdn$ holds in general.

Fix $\ell \in \N$. Define
$$
\Xl=\setc*{\! (z_1, \ldots,z_\ell)\in \C^\ell}{z_i\not=z_j \,\, \mbox{for any $i\not=j$} \!},
$$
called the \emph{configuration space} of $\ell$ distinct points on $\C^\ell$.
For $\z\in \Xl$, we have a superalgebra homomorphism
$$
 \Psi^\z :  U(\fgm)   \longrightarrow U(\mg)^{\otimes \ell} \blb z^{-1} \brb
$$
given by
 $$
  \Psi^\z(A[-r])=\sum_{i=1}^\ell \frac{A^{(i)}}{(z_i-z)^{r}}, \qquad \mbox{for $A \in \mg \,$ and $\, r \in \N$.}
 $$
 Hereafter $A^{(i)}:=\underbrace{1\otimes\cdots\otimes1\otimes \stackrel{i}{A}\otimes1\otimes\cdots\otimes1}_{\ell}$ for $i=1, \ldots, \ell$, and every rational function in $z$ represents its power series expansion at $\infty$.
The map $\Psi^\z$ extends to a superalgebra homomorphism
$$
 \Psi^\z :  U(\fgm) \blb \tau^{-1} \brb  \longrightarrow U(\mg)^{\otimes \ell} \zpz
$$
given by
 $$
 \Psi^\z  \left(\sum_{i=-\infty}^r a_i \tau^i \right)=\sum_{i=-\infty}^r \Psi^\z (a_i) \pz^i
 $$
for $a_i \in U(\fgm)$ and $r \in \Z$.

For $A \in \mg$, set
\begin{equation}\label{A(z)}
A(z)=\sum_{i=1}^\ell \frac{A^{(i)}}{z-z_i}.
\end{equation}
Consider the matrix
$$
\cLzs:= \Big[ \Psi^\z \big(\de_{i,j}\tau+(-1)^{2i}E_{i,j}[-1] \big)\Big]_{i,j \in \Is} =\Big[\de_{i,j}\pz-(-1)^{2i}E_{i,j}(z)\Big]_{i,j \in \Is}.
$$
Write $\cL^\z=\cLzs$ for simplicity.
Evidently, $\cL^\z$ is an amply invertible Manin matrix of type $\s$ over the superalgebra $U(\mg)^{\otimes \ell}  \zpz$.
The following proposition is a consequence of \propref{HM}.

\begin{prop} \label{Ber-sig-L}

$\Ber^{\s^\si} \left(\big(\cL^\z\big)^\si \right)=\Bers \! \left( \cL^\z\right)$ for any $\si \in \fS_\Is$.
\end{prop}

By \eqref{Bers-T},
$$
\Bers \! \left( \cL^\z\right)=\sum_{i=-\infty}^{p+m-q-n} b_i^\z \pz^i, \qquad\mbox{where} \quad b_i^\z:= \Psi^\z (b_i) \in  U(\mg)^{\otimes \ell} \blb z^{-1} \brb.
$$
Let $\cA_\pdn^\z$ be the subalgebra of $U(\mg)^{\otimes \ell}$ generated by the coefficients of the series $b_i^\z$, for $i \in \Z$ with $i \le p+m-q-n$.
It is called the \emph{Gaudin algebra for $\mg$}.

Write
$
\Az=\cA_\pdn^\z.
$
We remark that $\Az$ equals the subalgebra of $U(\mg)^{\otimes \ell}$ generated by the coefficients of $\Psi^\z(S)$ for $S \in \wh{\fz}_\pdn$ (as $\Psi^\z ( \T(S) )=d/dz ( \Psi^\z(S))$).
According to \cite[Corallary 3.6]{MR}, the algebra $\Az$ is a commutative subalgebra of $U(\mg)^{\otimes \ell}$.
It also commutes with the diagonal action of $\mg$.
In particular, $\Az$ acts on any weight space and any $\si\mm$singular space of the $\ell$-fold tensor product of arbitrary $\mg$-modules for any $\si \in \ov\fS_\Is$.

For any $\Az$-module $V$, the \emph{Gaudin algebra for $V$}, denoted by $(\Az)_V$, is defined to be the image of $\Az$ in $\End(V)$.
Suppose that $V$ is an $\Az$-module and $v$ is an eigenvector of $(\Az)_V$. Then there exists $c_i \in \C \blb z^{-1} \brb$ such that $b_i^\z v=c_i v$ for $i \le p+m-q-n$.
We call
$$
\cD_{\tns v}:= \sum_{i=-\infty}^{p+m-q-n} c_i  \pz^i
$$
the pseudo-differential operator associated to $\big(\Bers \! \left( \cL^\z\right), v \big)$.
It is obvious that $\Bers \! \left( \cL^\z\right) v=\cD_{\tns v} v$.

The isomorphism $\iota: \mg\oplus \C K \longrightarrow\og$, given by \eqref{iso-e}, restricts to an embedding $\iota: \mg \longrightarrow\og$, which induces a superalgebra homomorphism
$\iota: U(\mg) \longrightarrow U(\og)$ and hence a superalgebra homomorphism
$$
\uiot:=\iota^{\otimes \ell} : U(\mg)^{\otimes \ell} \longrightarrow U(\og)^{\otimes \ell}.
$$
The map $\uiot$ extends naturally to a superalgebra homomorphism
$$
\uiot : U(\mg)^{\otimes \ell}\zpz \longrightarrow U(\og)^{\otimes \ell}\zpz
$$
such that $\uiot (z^r)=z^r$ and $\uiot (\pz^r)=\pz^r$ for $r \in \Z$.
Consider the matrix
$$
\ocLzs:= \Big[\uiot \big(\de_{i,j}\pz-(-1)^{2i}E_{i,j}(z) \big) \Big]_{i,j \in \Is}=\Big[ \de_{i,j}\pz-(-1)^{2i} \oE_{i,j}(z) +\de_{i,j}\{i\}K(z)\Big]_{i,j \in \Is},
$$
where $\oE_{i,j}(z)$ and $K(z)$ are defined as in \eqref{A(z)}, and
$$
\{i\}:=\begin{cases}
\, 1 & \,\, \mbox{if} \ \  i \in \Ipq,\\
\, 0 & \,\, \mbox{if} \ \  i \in \Imn.
\end{cases}
$$
Write $\ov\cL^\z=\ocLzs$.
It is clear that $\uiot\big( \Bers ( \cL^\z)\big)=\Bers  ( \ov\cL^\z )$.
As in \propref{Ber-sig-L}, we have:

\begin{prop} \label{Ber-sig-hL}
$\Ber^{\s^\si} \left(\big(\ov\cL^\z \big)^\si \right)=\Bers \! \left( \ov\cL^\z \right)$ for any $\si \in \fS_\Is$.
\end{prop}

\begin{prop} \label{A-iso}
Let $M=M_1 \otimes \cdots \otimes M_\ell$, where $M_i \in \Os$ for all $i=1, \ldots, \ell$, and let $\si \in \ov \fS_\Is$.
For any singular weight $\ov\xi$ of $M$, the spaces $M_{\ov\xi}^{\sing}$ and $M_{\ov\xi^\si}^{\si\mm\sing}$ are $\Az$-modules.
Moreover, the map $\varphi: M_{\ov\xi}^{\sing}\longrightarrow M_{\ov\xi^\si}^{\si\mm\sing}$, given in \propref{Lth}, is an $\Az$-isomorphism.
\end{prop}

\begin{proof}
As $\ov E_{i,j}v=E_{i,j}v$ for all $v\in M$ and $i\not=j$, the proposition follows from \propref{Lth} and the fact that the action of $\Az$ commutes with the action of $\mg$.
\end{proof}

\section{Gaudin algebras for unitarizable modules} \label{Gaudin-Uni}

In this section, we first review the results of \cite{ChL25} on the Gaudin algebra for the general linear Lie (super)algebra over a tensor product of irreducible polynomial modules.
Subsequently, we extend these results from the case of irreducible polynomial modules to that of unitarizable highest weight modules.

\subsection{Tensor products of polynomial modules} \label{poly}
For $m, n \in \Z_+$ that are not all zero, let $\mg_\mn=\mg_{(0,n|m,0)} \cong \gl_\mn$, $\mb_\mn=\mb_{(0,n|m,0)}$, and $\mh_\mn=\mh_{(0,n|m,0)}$.

Let $\Js=\{-n+\hf, \ldots, -\thf, -\hf \} \cup \{ 1, 2, \ldots, m \}$.
Define $\theta \in \fS_\Js$ by sending
$$
  1, \ldots, m, -n+\hf, \ldots, -\hf \quad \mbox{ to}\quad -n+\hf, \ldots, -\hf, 1, \ldots, m,
$$
   respectively.
Then $\mb_\mn^\theta$ is the Borel subalgebra of $\mg_\mn$ with respect to $<_{(\Js, \theta)}$.
Write
$$
\Az_\mn=\Az_{(0,n|m,0)}.
$$
Recall that $\Pmn$ denotes the set of all $(\mn)$-hook partitions.
For any $\la \in \Pmn$,
let
$$
\ovla^\mn=\sum_{i=1}^{m}\la_i\ep_i+\sum_{i=1}^{n} \left\langle \la_i^\prime-m \right\rangle\ep_{-n+i-\hf}  \in \mh_\mn^{*}
$$
and
\begin{equation} \label{weight-neg}
{\mf X}^+_\mn=\setc*{\! \ovla^\mn \in \mh_\mn^{*}}{\la \in \Pmn \!},
\end{equation}
whose elements are dominant weights with respect to $\mb_\mn^\theta$.

Define
$$
\Xi_\mn=\sum_{i=1}^{m} \Zp \ep_i+\sum_{i=1}^{n} \Zp \ep_{-n+i-\hf}.
$$
A $\mg_\mn$-module $M$ is called a \emph{polynomial module} if $M$ is $\mh$-semisimple and every weight of $M$ belongs to $\Xi_\mn$.
Let $\mc{C}_\mn$ denote the category of polynomial $\mg_\mn$-modules; the morphisms in $\mc{C}_\mn$ are $\mg_\mn$-homomorphisms.
Also, every irreducible $\mg_\mn$-module in $\mc{C}_\mn$ is finite-dimensional (see, e.g., \cite[Proposition 2.2]{CW}).

The following proposition is well known (see, e.g., \cite[Theorem 3.26 and Theorem 3.27]{CW}) and also the proof of \cite[Theorem 6.2.2]{Lus}).

\begin{prop}\label{Cmn}
The category $\mc{C}_\mn$ is a semisimple tensor category, and each polynomial $\mg_\mn$-module $M$ decomposes into a direct sum of irreducible $\mg_\mn$-modules of the forms $L^\theta(\mg_\mn,\xi)$ with $\xi \in {\mf X}^+_\mn$.
 \end{prop}

\begin{rem}
Note that $\la \in \Pmn$ if and only if  $\la^\prime \in \cP_{n|m}$.
For any $\la \in \Pmn$,
let
\begin{equation} \label{weight-wtla}
\wt\la^\mn= \sum_{j=1}^{n} \la^\prime_{j}  \ep_{-n+j-\hf}   +\sum_{j=1}^m\langle\la_j-n\rangle\, \ep_j.
\end{equation}
For any $\xi:=\wt\la^\mn$ with $\la \in \Pmn$, we have $\xi^\theta=\ovla^\mn$ by \cite[Theorem 2.55]{CW}, and hence
\begin{equation} \label{L-theta}
L(\mg_\mn,\wt\la^\mn)=L^\theta(\mg_\mn,\ovla^\mn).
\end{equation}
Therefore, $L(\mg_\mn,\xi)$  is an irreducible polynomial $\mg_\mn$-module if and only if $\xi=\wt\la^\mn$ for some $\la\in \Pmn$.
\end{rem}

A  finite-dimensional commutative associative unital algebra $\cA$ is called a \emph{Frobenius algebra} if there is a nondegenerate symmetric bilinear form $(\cdot, \cdot)$ on $\cA$ such that
$$
(ab, c)=(a, bc) \quad \mbox{for all $a, b, c \in \cA$.}
$$

\begin{lem} [{\cite[Lemma 2.7]{Lu20}}] \label{cyc-Frob}
Let $\cA$ be a commutative associative unital algebra, and let $V$ be a finite-dimensional cyclic $\cA$-module.
Suppose that $V$ admits a nondegenerate symmetric bilinear form $\langle \cdot, \cdot \rangle$ with respect to which $\cA$ is symmetric (i.e., $\langle av, w \rangle=\langle v, aw \rangle$ for all $a \in \cA$ and $v, w \in V$), then the image of $\cA$ in $\End(V)$ is a Frobenius algebra.
\end{lem}

For $i=1, \ldots, \ell$, let $L_i$ be an irreducible polynomial $\mg_\mn$-module.
Define
$$
\mL=L_1 \otimes \cdots \otimes L_\ell.
$$
The following theorem has been established in \cite[Theorem 4.7 and Theorem 4.10]{ChL25}.
The nondegenerate symmetric bilinear form below is given in the proof of \cite[Theorem 4.7]{ChL25}.

\begin{thm}\label{ChL25}

For $\z \in \Xl$, we have:

\begin{enumerate}[\normalfont(i)]
\item $\mL^{\theta\mm\sing}$ is a cyclic $\Az_\mn$-module.

 \item $\mL^{\theta\mm\sing}$ admits a nondegenerate symmetric bilinear form  with respect to which $\Az_\mn$ is symmetric. Moreover, the bilinear form is nondegenerate on every $\theta\mm$singular weight space of $\mL$.

\item $(\Az_\mn)_{\mL^{\theta\mm\sing}}$ is a Frobenius algebra.

\item $(\Az_\mn)_{\mL^{\theta\mm\sing}}$ is diagonalizable with a simple spectrum for generic $\z \in \Xl$.

\end{enumerate}

\end{thm}

\begin{cor} \label{ChL25-poly-weight}

For any $\z \in \Xl$ and any $\theta$-singular weight $\mu$ of $\mL$, we have:

\begin{enumerate}[\normalfont(i)]
\item $\mL_\mu^{\theta\mm\sing}$ is a cyclic $\Az_\mn$-module.

\item $(\Az_\mn)_{\mL_\mu^{\theta\mm\sing}}$ is a Frobenius algebra.

\item $(\Az_\mn)_{\mL_\mu^{\theta\mm\sing}}$ is diagonalizable with a simple spectrum for generic $\z \in \Xl$.

\end{enumerate}

\end{cor}

\begin{proof}
  Let $\mu$ be a $\theta$-singular weight of $\mL$.
	Since $\Az_\mn$ preserves every $\theta$-singular weight space, (i) and (iii) follow from (i) and (iv) of \thmref{ChL25}.
  Moreover, (ii) follows from \thmref{ChL25}(ii) and \lemref{cyc-Frob}.
\end{proof}

Analogous to \propref{A-iso}, we have the following.

\begin{prop} \label{A-iso-poly}

For any $\z \in \Xl$ and any singular weight $\xi$ of $\mL$, the spaces $\mL_{\xi}^{\sing}$ and $\mL_{\xi^\theta}^{\theta \mm\sing}$ are $\Az_\mn$-modules.
Moreover, there is an element $X \in \mg_\mn$, which is a product of elements in $\setc*{E_{i,j} }{  i,j \in \Js, \, i+j \notin \Z}$, such that the map
$$
\vartheta: \mL_{\xi^\theta}^{\theta \mm\sing} \longrightarrow \mL_{\xi}^{\sing},
$$
defined by $\vartheta(v)=X v$ for $v \in  \mL_{\xi^\theta}^{\theta \mm\sing}$, is an $\Az_\mn$-isomorphism
\end{prop}

By applying \propref{A-iso-poly} to \corref{ChL25-poly-weight}, we obtain the following.

\begin{cor} \label{ChL25-weight}

For any $\z \in \Xl$ and any singular weight $\xi$ of $\mL$, we have:

\begin{enumerate}[\normalfont(i)]
\item $\mL_\xi^{\sing}$ is a cyclic $\Az_\mn$-module.

\item $(\Az_\mn)_{\mL_\xi^{\sing}}$ is a Frobenius algebra.

\item $(\Az_\mn)_{\mL_\xi^{\sing}}$ is diagonalizable with a simple spectrum for generic $\z \in \Xl$.

\end{enumerate}

\end{cor}

\subsection{Tensor products of unitarizable modules} \label{uni}

Recall that $\mg:=\gs$ and $\og:=\ogs$.
For each $i=1, \ldots, \ell$, let $L_i=L(\mg, \xi_i)$ for some $\xi_i \in {\mf X}^{d_i}_\pdn$ (see \eqref{un-weight}), where $d_i \in \N$, and let
$$
\mL =L_1 \otimes \cdots \otimes  L_\ell.
$$
Recall the isomorphism $\iota: \mg\oplus \C K \longrightarrow \og$ defined by \eqref{iso-e}.
By \propref{CLZ} and \propref{uni-K}, $L_1,\ldots, L_\ell, \mL$ are all unitarizable, and $L_i=L(\og, \ov\xi_i)^\iota$ for some $\ov\xi_i \in \ov{\mf X}^d_\pdn$.

 For each $i=1, \ldots, \ell$, let $\ov L_i=L(\og, \ov\xi_i)$ and
 $$
\omL =\ov L_1 \otimes \cdots \otimes  \ov L_\ell.
$$
Then $\omL\in \Os$ and $\mL=\omL^\iota$.

The space $\mL$ decomposes into a direct sum of unitarizable highest weight $\mg$-modules.
More precisely,
$$
\mL= \bigoplus_{\xi} L(\mg, \xi)^{\oplus m_\xi}.
$$
The direct sum is taken over all $\xi  \in {\mf X}^{d}_\pdn$, where $d:=\sum_{i=1}^\ell d_i$, and $m_\xi \in \Zp$ are the multiplicities of $L(\mg, \xi)$ in the decomposition (see \cite[Section 6]{CLZ}).
By the discussion preceding \propref{uni-K}, we see that $\omL= \bigoplus_{\ov\xi}L(\og, \ov\xi)^{\oplus m_\xi}$ as $\og$-modules, where the direct sum is taken over all $\ov\xi  \in \ov {\mf X}^{d}_\pdn$.
Note that $L(\mg, \xi)=L(\og, \ov\xi)^\iota$ for all such $\ov\xi$.

We now state the main result of the paper, which will be proved in \secref{pf}.

\begin{thm}\label{Main}

For any $\z \in \Xl$ and any singular weight $\xi$ of $\mL$, we have:

\begin{enumerate}[\normalfont(i)]
\item $\mL_\xi^{\sing}$ is a cyclic $\cA_\pdn^\z$-module.

\item $(\cA_\pdn^\z)_{\mL_\xi^{\sing}}$ is a Frobenius algebra.

\item $(\cA_\pdn^\z)_{\mL_\xi^{\sing}}$ is diagonalizable with a simple spectrum for generic $\z \in \Xl$.

\end{enumerate}

\end{thm}

\begin{cor} \label{max}
For any $\z \in \Xl$ and any singular weight $\xi$ of $\mL$, the following properties hold:

\begin{enumerate} [\normalfont(i)]

\item The algebra $(\cA_\pdn^\z)_{\mL^\sing_{\xi}}$ is a maximal commutative subalgebra of $\End(\mL^\sing_\xi)$ of dimension $\dim \! \left({\mL^\sing_\xi} \right)$.

\item Every eigenspace of the algebra $(\cA_\pdn^\z)_{\mL^\sing_{\xi}}$ is one-dimensional, and the set of eigenspaces of $(\cA_\pdn^\z)_{\mL^\sing_{\xi}}$ is in one-to-one correspondence with the set of maximal ideals of $(\cA_\pdn^\z)_{\mL^\sing_{\xi}}$.

\item Every generalized eigenspace of $(\cA_\pdn^\z)_{\mL^\sing_{\xi}}$ is a cyclic $\cA_\pdn^\z$-module.

\end{enumerate}
\end{cor}

\begin{proof}
This follows from \thmref{Main} and \cite[Lemma 1.3]{Lu20}.
\end{proof}

We will present a procedure for constructing an eigenbasis for $(\cA_\pdn^\z)_{\mL^\sing_{\xi}}$ in \secref{eig}.
A description of the corresponding eigenvalues will also be obtained.

\section{Proof of the main result} \label{pf}

Fix $\s:=(0^p, 1^q, 0^m, 1^n)$. For any sufficiently invertible matrix $A$ of type $\s$, we write
\begin{equation} \label{drop-s}
\Ber(A)=\Bers(A).
\end{equation}
We will recover the superscript $\s$ if necessary.
Let us keep the notations of \secref{uni}, and let $\z \in \Xl$.

\begin{lem}\label{wta}
Write $(\ocLzs)^{-1}=[\tilde{a}_{ij}]_{i,j \in \Is}$, $(\ocLz_{(p-1,q|m,n)})^{-1}=[\tilde{a}^\prime_{ij}]_{i,j \in \Is\backslash\{-p\}}$, and $(\ocLz_{(p,q-1|m,n)})^{-1}=[\tilde{a}^{\pp}_{ij}]_{i,j \in \Is\backslash\{-q+\hf\}}$. Let $v$ be a weight vector of weight $\mu$ in $\omL$.

\begin{enumerate} [\normalfont(i)]

\item If  $\mu \in \Xi_{(p-1,q|m,n)}$, then $\tilde{a}_{ij}v=\tilde{a}^\prime_{ij}v$ for all $i,j \in \Is\backslash\{-p\}$.

\item If  $\mu \in \Xi_{(p,q-1|m,n)}$, then $\tilde{a}_{ij}v=\tilde{a}^\pp_{ij}v$ for all $i,j \in \Is\backslash\{-q+\hf\}$.

\end{enumerate}
\end{lem}

\begin{proof}

We only prove (i). The proof of (ii) is similar. 
For simplicity, we suppress $\z$.
Let $\ocL=\ocLs$ and $\ocL^\prime=\ocL_{(p-1,q|m,n)}$ .
We have $\ocL=[a_{ij}]_{i,j \in \Is}$, where
$$
a_{ij}=\de_{i,j}\pz-(-1)^{2i} \oE_{i,j}(z) +\de_{i,j}\{i\} K(z).
$$
It is straightforward to verify that
\begin{equation} \label{L-inv}
\ocL^{-1}=D^{-1} \sum_{r=0}^\infty \left( {\bf 1}   -\ocL D^{-1} \right)^r.
\end{equation}
where $\bf 1$ is the $(p+q+m+n) \times (p+q+m+n)$ identity matrix, and $D:=[\de_{i,j}\pz]_{i,j \in \Is}$.

We regard $U(\og)^{\otimes \ell} \zpz$ as a $\og$-module via the adjoint action.
For any $i, j \in \Is$, $a_{ij}$ is of weight $\ep_i-\ep_j$.
By \eqref{L-inv}, $\ta_{ij}$ is also of weight $\ep_i-\ep_j$.
On the other hand, we regard $\omL \zpz$ as a $\og$-module in the obvious way.
Any weight vector in $\omL \zpz$ has weight contained in $\Xis$.

Obviously, $\ocL^\prime=[a_{ij}]_{i,j \in \Is\backslash\{-p\}}$.
Analogous to \eqref{L-inv}, we have
\begin{equation} \label{L-inv-prime}
(\ocL^\prime)^{-1}= (D^\prime)^{-1} \sum_{r=0}^\infty \left( {\bf 1}^\prime   -\ocL^\prime (D^\prime)^{-1} \right)^r,
\end{equation}
where $\bf 1^\prime$ is the $(p+q+m+n-1) \times (p+q+m+n-1)$ identity matrix, and $D^\prime:=[\de_{i,j}\pz]_{i,j \in  \Is\backslash\{-p\}}$. 
For $r \in \Zp$, write
$$
\left( {\bf 1}   -\ocL D^{-1} \right)^r=[b_{i,j}]_{i,j \in \Is} \qquad \text{and} \qquad \left( {\bf 1}^\prime   -\ocL^\prime (D^\prime)^{-1} \right)^r=[b^\prime_{i,j}]_{i,j \in \Is\backslash\{-p\}}.
$$
We claim that $b_{i,j} v=b^\prime_{i,j} v$ for any $r \in \Zp$ and $i,j \in \Is\backslash\{-p\}$. 
This proves (i) by \eqref{L-inv} and \eqref{L-inv-prime}.

We are left to prove the claim.
It is clear if $r=0$. 
Suppose now that $r \in \N$.
Let
$$
c_{ij}=\de_{i,j}-a_{i,j} \pz^{-1} \qquad \mbox{for $\, i,j \in \Is$}.
$$
Fix $i,j \in \Is\backslash\{-p\}$. The vector $b_{i,j} v$ is the sum of the vectors
\begin{equation} \label{a-vec}
c_{k_1,k_2} c_{k_2,k_3} \ldots c_{k_r,k_{r+1}}v,
\end{equation}
where $k_1=i$, $k_{r+1}=j$, and $k_2, \ldots, k_r$ run over all elements of $\Is$.
Similarly, the vector $b^\prime_{i,j} v$ is the sum of the vectors in \eqref{a-vec}, but with $k_2, \ldots, k_r$ running over all elements of $\Is\backslash\{-p\}$.
Let $w$ be the vector in \eqref{a-vec} such that $k_s=-p$ for some $s=2, \ldots, r$.
Since the vector $c_{k_s,k_{s+1}} c_{k_{s+1},k_{s+2}} \ldots c_{k_r,k_{r+1}}v$ is zero, for otherwise it would have weight $(\ep_{-p}-\ep_j)+\mu\notin\Xis$,
we have $w=0$.
This implies that $b_{i,j} v=b^\prime_{i,j} v$, as claimed.
\end{proof}

The following proposition is crucial to proving \thmref{Main}.

\begin{prop} \label{trunc}
Let $c(z)=\sum_{i=1}^\ell d_i (z-z_i)^{-1}$, and let $p', q', m', n' \in \Z_+$ satisfying $p' \le p$, $q' \le q$, $m' \le m$, and $n' \le n$.
Let $v$ be a weight vector of weight $\mu$ in $\omL$.

\begin{enumerate} [\normalfont(i)]

\item   If $\mu \in \Xi_{(p',q|m,n)}$, then
$$
\Ber \big(\cLzs\big) v= \big(\pz+c(z) \big)^{p-p'} \, \Ber \big(\cLz_{(p',q|m,n)} \big)  v.
$$

\item If  $\mu \in \Xi_{(p,q'|m,n)}$, then
$$
\Ber \big(\cLzs\big) v=  \big(\pz+c(z) \big)^{q'-q} \, \Ber \big(\cLz_{(p,q'|m,n)} \big)  v.
$$

\item If  $\mu \in \Xi_{(p,q|m',n)}$, then
$$
\Ber \big(\cLzs\big) v=  \Ber \big(\cLz_{(p,q|m',n)} \big) \pz^{m-m'} v.
$$

\item If  $\mu \in \Xi_{(p,q|m,n')}$, then
$$
\Ber \big(\cLzs\big) v=   \Ber \big(\cLz_{(p,q|m,n')} \big) \pz^{n'-n} v.
$$

\end{enumerate}

\end{prop}

\begin{proof}
We suppress $\z$ and continue with the notation of \lemref{wta} and its proof. 
Let $\cL=\cLs$ and $\ocL=\ocLs$.

It is enough to prove (i) for the case $p'=p-1$ (as the general case will follow by induction).
Let $\mu \in \Xi_{(p-1,q|m,n)}$.
Choose $\pi \in \fS_\Is$ such that $\Is^\pi$ has the total order
\begin{eqnarray*}
& \hspace{-6cm} -p\lepi \cdots\lepi -1\lepi 1\lepi \cdots \lepi m \\
&  \lepi -q+\hf\lepi\cdots\lepi -\hf \lepi \hf \lepi\cdots\lepi n-\hf.
\end{eqnarray*}
Then $\ocL^\pi$ is a matrix of type $\s^\pi:=(0^{p+m}, 1^{q+n})$.
By \propref{Ber-sig-hL}, $\Ber \big(\ocL\big)=\Ber^{\s^\pi} \big(\ocL^\pi \big)$. 
Clearly, $(\ocL^\pi)^{-1}=(\ocL^{-1})^\pi$.
Using the formula for $\Ber^{\s^\pi} \big(\ocL^\pi \big)$ in \cite[Definition 2.6]{MR},
we see that $\Ber \big(\ocL\big)$ equals
\begin{eqnarray}
& \hspace{-4cm} \Big(\sum_{\si \in \fS_{\Gamma}} (-1)^{l(\si)}  a_{\si(-p),-p}\cdots a_{\si(-1),-1}a_{\si(1),1} \cdots a_{\si(m),m} \Big)  \label{MoDefn-1} \\
&  \cdot \Big(\sum_{\tau \in \fS_{\ov \Gamma}} (-1)^{l(\tau)} \ta_{-q+\hf,\tau(-q+\hf)}\cdots \ta_{-\hf,\tau(-\hf)} \ta_{\hf,\tau(\hf)} \cdots \ta_{n-\hf,\tau(n-\hf)}\Big). \label{MoDefn-2}
\end{eqnarray}
Here $\fS_{\Gamma}$ and $\fS_{\ov\Gamma}$ denote the symmetric groups on the sets
$\Gamma=\{ -p, \ldots, -1\} \cup \{ 1, \ldots, m\} $
and
$\ov\Gamma=\left\{-q+\hf, \ldots, -\hf \right\}  \cup \left\{\hf, \ldots, n-\hf \right\}$, respectively.
Let $B$ be the sum in \eqnref{MoDefn-2}, and let $\si \in \fS_{\Gamma}$.
The vector $w:=a_{\si(-p+1),-p+1}\cdots a_{\si(-1),-1}a_{\si(1),1} \cdots a_{\si(m),m} Bv$, if nonzero, has weight $-(\ep_{\si(-p)}-\ep_{-p})+\mu$.
If $\si(-p)\not=-p$, then $-(\ep_{\si(-p)}-\ep_{-p})+\mu\notin\Xis$, and hence $w=0$.
Thus, only those $\si \in \fS_{\Gamma}$ with $\si(-p)=-p$ in the summation of \eqref{MoDefn-1} will contribute to $\Ber \big(\ocL\big) v$.
Suppose now that $\si(-p)=-p$.
As $w$ is a sum of vectors of the form $(v_1 \otimes \ldots \otimes v_\ell) z^r\pz^s$, where $r, s\in \Z$ and $v_k$ is some weight vector of $\ov L_k$ with weight $\mu_k$ satisfying $\sum_{k=1}^\ell \mu_k = \mu$, we see by \lemref{weight=0} that each $\mu_k  \in \Xi_{(p-1,q|m,n)}$, and so $\oE_{-p,-p}^{(k)}(v_1 \otimes \ldots \otimes v_\ell z^r\pz^s) =\mu_k(\oE_{-p,-p}) (v_1 \otimes \ldots \otimes v_\ell) z^r\pz^s =0$ for all $k$.
Therefore, $\oE_{-p,-p}(z)w=0$, and any such $\si$ contributes the term $(\pz+K(z)) w=(\pz+c(z)) w$ to $\Ber \big(\ocL\big) v$.
By \lemref{wta}(i), we conclude that
$$
\Ber \big(\ocL\big) v=\big(\pz+c(z)\big) \Ber \big(\ocL_{(p-1,q|m,n)}\big) v.
$$
By definition, $\Ber \big(\cL) v=\Ber \big(\ocL\big) v$ and $\Ber \big(\cL_{(p-1,q|m,n)} \big)  v=\Ber \big(\ocL_{(p-1,q|m,n)}\big) v$.
This proves (i) for $p'=p-1$.

Similarly, it suffices to prove (ii) for the case $q'=q-1$.
Let $\mu \in \Xi_{(p,q-1|m,n)}$.
Choose $\pi \in \fS_\Is$ such that $\Is^\pi$ has the total order
\begin{eqnarray*}
&\hspace{-6cm} -p\lepi \cdots\lepi -1\lepi 1\lepi\cdots\lepi m \\
&\hspace{1cm} \lepi -q+\frac{3}{2}\lepi\cdots\lepi -\hf \lepi  \hf \lepi\cdots\lepi n-\hf \lepi -q+\hf.
\end{eqnarray*}
Then $\ocL^\pi$ is a matrix of type $\s^\pi:=(0^{p+m}, 1^{q+n})$.
Using \propref{Ber-sig-hL} and the formula for $[\Ber^{\s^\pi} (\ocL^\pi)]^{-1}$ in \cite[Corollary 2.16]{MR}, we find that $[\Ber \big(\ocL\big)]^{-1}$ equals
\begin{eqnarray*}
& \hspace{-4.5cm} \Big(\sum_{\si \in \fS_{\Gamma}} (-1)^{l(\si)} \ta_{\si(-p),-p}\cdots \ta_{\si(-1),-1} \ta_{\si(1),1} \cdots \ta_{\si(m),m} \Big) \\
& \hspace{0.5cm} \cdot\Big(\sum_{\tau \in \fS_{\ov \Gamma}} (-1)^{l(\tau)} a_{-q+\frac{3}{2},\tau(-q+\frac{3}{2})}\cdots a_{-\hf,\tau(-\hf)} a_{\hf,\tau(\hf)} \cdots a_{n-\hf,\tau(n-\hf)} a_{-q+\hf,\tau(-q+\hf)} \Big).
\end{eqnarray*}
 Let $X=\Ber \big(\ocL\big)$, $Y=\Ber \big(\ocL_{(p,q-1|m,n)} \big)$, and $\gam=\pz+c(z)$.
By \lemref{wta}(ii) and an argument similar to the proof of (i), we have
$$
X ^{-1}w=   Y^{-1} \gam w \qquad \mbox{ for all $w \in \ov \mL_\mu$.}
$$
Consequently,
$$
\gam X v= \sum_{r=0}^\infty \big(1-X^{-1}\gam^{-1}\big)^r v= \sum_{r=0}^\infty (1-Y^{-1})^r v =Y v,
$$
where the second equality follows from the fact that each coefficient of $z^i\pz^j$ in $X^{-1}$, which belongs to $\Az_\pdn$, preserves any weight space.
This proves (ii) for $q'=q-1$.

As to (iii) (resp., (iv)), proving the case $m'=m-1$ (resp., $n'=n-1$) suffices.
This can be achieved by the same argument as in the proof of \cite[Proposition 3.16]{ChL25} for $m'=m-1$ (resp., $n'=n-1$) there.
We leave the details to the reader.
\end{proof}

For any singular weight $\xi$ of $\mL$, we have $\mL^\sing_{\xi}=\omL^\sing_{\ov\xi}$ for some $\ov\xi  \in \ov {\mf X}^{d}_\pdn$.
We will restrict our attention to $\ov\mL$.
For $i=1,\ldots,\ell$, write
$$
\ov\xi_i=\ovlas_{(i)} \qquad \text{and} \qquad  \ov\xi=\ovlas,
$$
where $\la_{(i)}\in \oP^{d_i}_\pdn$ and $\la\in \oP^d_\pdn$.
Choose $r \in \N$ large enough that $m\le r$, $q\le r$, $(\la^+)^\prime_1 \le r$ and $\la^-_1 \le r$, and  $(\la_{(i)}^+)^\prime_1 \le r$ and $(\la_{(i)}^-)_1 \le r$ for all $i=1, \ldots, \ell$.
Obviously, $\oP^k_\pdn \subseteq \oP^k_{(p, r|r, n)}$ for all $k \in \N$.
For $i=1,\ldots,\ell$, let $\ov L^\prime_i=L(\og_{(p, r|r, n)}, \ovla_{(i)}^{(p, r|r, n)})$ and
$$
\ov\mL^\prime=\ov L^\prime_1\otimes \cdots \otimes\ov L^\prime_\ell.
$$
By \propref{trun-sep}, $\tr_\pdn(\ov L^\prime_i)= \ov L_i$.
This means that $\tr_\pdn(\ov\mL^\prime)=\ov\mL$.
According to \exref{ex}, $(\ovla_{(i)}^{(p, r|r, n)})^\sig=\ov\xi_i$ and $(\ovla^{(p, r|r, n)})^\sig=\ov\xi$, where $\sig:=\sig_{(q,m)}$ is defined as in \eqnref{si-order}.
Furthermore, $(\ov\mL^\prime)_{\ov\xi}^{\sig\mm\sing}=\ov\mL_{\ov\xi}^\sing$.

For $i=1,\ldots,\ell$, it is easy to see that $\ov L^\prime_i$ has the highest weight
$$
 \ovla_{(i)}^{(p, r|r, n)}=- \sum_{j=1}^{r} (\la^{-}_{(i)})^\prime_{j} \ep_{-j+\hf}   +\sum_{j=1}^r (\la^{+}_{(i)} )_j\, \ep_j+d_i\La_0.
$$
That is, $ \ovla_{(i)}^{(p, r|r, n)}= \ovla_{(i)}^{(0, r|r, 0)}$, and we may regard $\ovla_{(i)}^{(p, r|r, n)} \in \ov\fh^{*}_ {(0, r|r, 0)}$, which is the highest weight of $\tr_\orr(\ov L^\prime_i)$.
Therefore, $\tr_\orr(\ov L^\prime_i)=\ov L(\og_\orr,\ovla_{(i)}^\orr)$.
Similarly,
$$
 \ovla^{(p, r|r, n)}=- \sum_{j=1}^{r} (\la^{-})^\prime_{j} \ep_{-j+\hf}   +\sum_{j=1}^r (\la^{+} )_j\, \ep_j+d \La_0
$$
and $ \ovla^{(p, r|r, n)}= \ovla^{(0, r|r, 0)}$.

For $i=1,\ldots,\ell$, let $\mu_{(i)}$ be the partition whose conjugate partition is
$$
\mu_{(i)}^\prime=\big(d_i-(\la^{-}_{(i)})^\prime_r,\ldots,d_i-(\la^{-}_{(i)})^\prime_1, (\la^{+}_{(i)})^\prime_1,\ldots,(\la^{+}_{(i)})^\prime_r \big).
$$
Note that $\mu_{(i)}$ is indeed a partition since $(\la^{-}_{(i)})^\prime_1+ (\la^{+}_{(i)})^\prime_1 \le d_i$, and that by the choice of $r$, $\mu_{(i)} \in \cP_\rr$.
Recall from \eqref{weight-wtla} that
$$
\wt\mu_{(i)}^\rr= \sum_{j=1}^{r} \big(d_i-(\la^{-}_{(i)})^\prime_{j} \big) \ep_{-j+\hf}   +\sum_{j=1}^r (\la^{+}_{(i)} )_j\, \ep_j.
$$
Similarly, let $\mu \in \cP_\rr$ be the partition whose conjugate partition is
$$
\mu^\prime=\big(d-(\la^{-})^\prime_r,\ldots,d-(\la^{-})^\prime_1, (\la^{+})^\prime_1,\ldots,(\la^{+})^\prime_r \big).
$$
Obviously,
$$
\wt\mu^\rr= \sum_{j=1}^{r} \big(d-(\la^{-})^\prime_{j} \big) \ep_{-j+\hf}   +\sum_{j=1}^r (\la^{+} )_j\, \ep_j.
$$

For $i=1,\ldots,\ell$, let
$$
L^\pp_i=L(\mg_\orr, \wt\mu_{(i)}^\rr) \qquad \text{and} \qquad \ov L^\pp_i=\tr_\orr(\ov L^\prime_i).
$$
Then $L^\pp_i=(\ov L^\pp_i)^\iota$, and $L^\pp_i$ is an irreducible polynomial $\mg_\mn$-module.

Define
$$
\mL^\pp= L_1^\pp\otimes\cdots\otimes  L^\pp_\ell  \qquad \text{and} \qquad \ov\mL^\pp=\ov L_1^\pp\otimes\cdots\otimes \ov L^\pp_\ell.
$$
It is clear that $\ov \mL^\pp=\tr_\orr(\ov \mL^\prime)$ and
\begin{equation} \label{wt-nu}
(\ov\mL^\pp)_{\ovla^{(0, r|r, 0)}}^\sing=(\mL^\pp)_{\wt\mu^\rr}^\sing.
\end{equation}

\vspace{3mm}

\noindent\emph{Proof of \thmref{Main}}:
Let $\z \in \Xl$. For ease of exposition, we say that a commutative algebra $\cB^\z$ is $\cfd$ on a $\cB^\z$-module $V$ provided that the properties (i), (ii), and (iii) of \thmref{Main} are satisfied if $\cA_\pdn^\z$ and $\mL_\xi^{\sing}$ are replaced by $\cB^\z$ and $V$, respectively.

By \corref{ChL25-weight} and \eqref{wt-nu},  $\cA_\orr^\z$ is $\cfd$ on  $(\ov\mL^\pp)_{\ovla^{(0, r|r, 0)}}^\sing$.
Since $M:=( \ov \mL^\prime)_{\ovla^{(p, r|r, n)}}^\sing=(\ov\mL^\pp)_{\ovla^{(0, r|r, 0)}}^\sing$, we have $(\cA_{(p,r|r,n)}^\z)_M=(\cA_\orr^\z)_M$ by \propref{trunc}.
Thus, $\cA_{(p,r|r,n)}^\z$ is $\cfd$ on $( \ov \mL^\prime)_{\ovla^{(p, r|r, n)}}^\sing$.

Recall that $\ov\xi=(\ovla^{(p, r|r, n)})^\sig$.
According to \propref{A-iso}, there exists $A \in U(\og_{(p, r|r,n)})$, which is a product of elements in $\setc*{ \oE_{i,j} }{i,j \in \Is^-_{p|r}\cup\Is^+_{r|n}, \, ij>0, \, i+j \notin \Z}$, such that the map
\begin{equation} \label{Arr-iso}
\varphi: (\ov\mL^\prime)_{\ovla^{(p, r|r, n)}}^\sing \longrightarrow   (\ov\mL^\prime)_{\ov\xi}^{\sig\mm\sing},
\end{equation}
defined by $\varphi(v)=A v$ for $v \in (\ov\mL^\prime)_{\ovla^{(p, r|r, n)}}^\sing$, is an $\cA^\z_{(p, r|r,n)}$-isomorphism.
Hence, $\cA_{(p,r|r,n)}^\z$ is $\cfd$ on $N:=(\ov\mL^\prime)_{\ov\xi}^{\sig\mm\sing}$.
Since $N=\ov\mL_{\ov\xi}^\sing$, $(\Az_{(p, q|m, n)})_N=(\Az_{(p, r|r, n)})_N$ in view of \propref{trunc}, and therefore $\Az_{(p, q|m, n)}$ is $\cfd$ on $N$.
This completes the proof of \thmref{Main}.
\qed

\section{Super Bethe ansatz} \label{Bethe-ans}

We use the notations of \secref{uni} and \secref{pf}.
Fix $\z \in \Xl$.
In this section, we show that an eigenbasis for the Gaudin algebra $\cA_\pdn^\z$ on $\mL^\sing_\xi$ can be constructed from any eigenbasis for the Gaudin algebra $\cA_\orr^\z$ on the $\theta$-singular weight space of $\mL^\pp$ determined by $\mL^\sing_\xi$.
This, together with the Bethe ansatz method, enables us to obtain the super Bethe ansatz for $\cA_\pdn^\z$ on $\mL^\sing_\xi$.

\subsection{Bethe ansatz method} \label{BAM}
We review the Bethe ansatz method of constructing eigenvectors for the Gaudin algebra on a tensor product of irreducible finite-dimensional $\gl_m$-modules (\cite{BF, FFR}).

Note that $\gl_m=\mg_{(0,0|m,0)}$.
Write $\cA_{m}^\z=\cA_{(0,0|m,0)}^\z$.
Let $V_k=L(\gl_m, \tau_k)$ for some dominant integral weight $\tau_k$ for $\gl_m$, and let $v_k$ be a highest weight vector of $V_k$ for $k=1, \ldots, \ell$.
Define
$$
\mV =V_1  \otimes \cdots \otimes  V_\ell \qquad \text{and} \qquad |0 \rangle=v_1 \otimes \ldots \otimes v_\ell \in \mV.
$$
Let $f_j=E_{j+1,j}$ for $j=1, \ldots, m-1$.
Given $i_1,\ldots,i_s \in \{1,\ldots, m-1 \}$ (not necessarily distinct) and any pairwise distinct $w_1,\ldots, w_s \in \C$ with $w_j \ne z_k$ for all $j,k$,
we define
$$
 \big|w_1^{i_1},\ldots,w_s^{i_s} \big\rangle =\sum_{(I^1,\ldots,I^\ell)} \prod_{k=1}^\ell  \frac{f_{i_{j^k_1}}^{(k)} \cdots f_{i_{j^k_{a_k}}}^{(k)}} {(w_{j^k_1}-w_{j^k_2}) \cdots (w_{j^k _{a_k}}-w_{j^k _{a_k+1}})} |0 \rangle   \in \mV .
$$
The summation is taken over all ordered partitions $I^1\cup I^2\cup \ldots \cup I^\ell$ of the set $\{1, \ldots, s \}$, where $I^k :=\big\{j^k _1,j^k _2,\ldots,j^k _{a_k } \big\}$ and $w_{j^k _{a_k+1}}:=z_k$ for $k=1, \ldots, \ell$.
The vector $\big|w_1^{i_1},\ldots,w_s^{i_s} \big\rangle$ is called a \emph{Bethe vector} in $\mV $.

Let $\{ \al_i:=\ep_i-\ep_{i+1} \, | \, i=1, \ldots, m-1 \}$ be the set of all simple roots of $\gl_m$, and let $\check\al_i=\Eii-E_{i+1,i+1}$ for $i=1, \ldots, m-1$.
The equations
\begin{equation} \label{BAn}
\sum_{k=1}^\ell\frac{\tau_k (\check \al_{i_j})}{w_j-z_k}-\sum_{\substack{r=1 \\ r \ne j}}^s \frac{\al_{i_r}(\check\al_{i_j})}{w_j-w_r}=0,  \qquad j=1,\ldots, s,
\end{equation}
are called the \emph{Bethe ansatz equations} (\cite{BF, FFR}).
If \eqref{BAn} hold and the vector $\big|w_1^{i_1},\ldots,w_s^{i_s} \big\rangle$ is nonzero, then $\big|w_1^{i_1},\ldots,w_s^{i_s} \big\rangle$ is an eigenvector of $(\Am)_{\mV^\sing_{\! \ga}}$, where $\ga:=\sum_{k=1}^\ell \tau_k-\sum_{r=1}^s \al_{i_r}$ (see \cite{MTV06, RV}).

The \emph{completeness of the Bethe ansatz} is a conjecture asserting that the Bethe vectors form an eigenbasis for $(\Am)_{\mV^\sing_{\! \ga}}$ for a generic $\z$.
In \cite{MV05}, positive evidence for the conjecture is given.
However, it is shown in \cite{MV07} that the Bethe ansatz equations have no solutions for some examples.

Let
$$
\cE_{\ns i}(z)=\sum_{k=1}^\ell \frac{\tau_k (\Eii)}{z-z_k}-\sum_{r=1}^s \frac{\al_{i_r}(\Eii)}{z-w_r}, \qquad i=1, \ldots, m.
$$
Set
$$
\bD_m=\big(\pz-\cE_{\ns 1}(z) \big) \cdots \big(\pz-\cE_{\ns m}(z) \big).
$$
The Gaudin algebra $\cA_{m}^\z$ is determined by $\Ber^{\s_0}(\cL^\z_{(0,0|m,0)})=\cdet(\cL^\z_{(0,0|m,0)})$, where $\s_0=(0^m)$ (see \propref{CFR}).
Suppose that the Bethe ansatz equations \eqref{BAn} are satisfied. Then
$$
\cdet(\cL^\z_{(0,0|m,0)}) \big|w_1^{i_1},\ldots,w_s^{i_s} \big\rangle=\bD_m  \big|w_1^{i_1},\ldots,w_s^{i_s} \big\rangle.
$$
(see \cite[Theorem 9.2]{MTV06} or \cite[Theorem 3.2]{MM17})(cf. \cite[Theorem 3]{FFR}).

\subsection{Eigenvectors and eigenvalues} \label{eig}

We keep the notations introduced in \secref{uni} and \secref{pf} unless otherwise stated.
We also use the convention \eqref{drop-s}.
We summarize the isomorphisms of vector spaces considered in \secref{uni} and \secref{pf} as follows:

\begin{equation}\label{diagram}
\begin{tikzcd}
(\ov\mL^\prime)_{\ov\xi}^{\sig\mm\sing} \arrow[d,swap, "\tr_\pdn"]
&\quad (\ov\mL^\prime)_{\ovla^{(p, r|r, n)}}^\sing\arrow[l,swap,"\varphi"]  \arrow[d,"\tr_\orr"]     &\\
\mL^\sing_{\xi}=\omL^\sing_{\ov\xi}\qquad&\qquad \qquad (\ov\mL^\pp)_{\ovla^{(0, r|r, 0)}}^\sing=(\mL^\pp)_{\wt\mu^\rr}^\sing& (\mL^\pp)^{\theta\mm\sing}_{\ov\mu^\rr}\arrow[l,swap,"\vartheta"]
\end{tikzcd}
\end{equation}
Here the $\cA^\z_{(p, r|r,n)}$-isomorphism $\varphi$ is the map in \eqref{Arr-iso},
and the $\cA_\orr^\z$-isomorphism $\vartheta$ is given by $\vartheta(v)=X v$ for some $X \in U(\mg_\orr)$ by \propref{A-iso-poly} and \eqref{L-theta}; 
furthermore, the vertical maps are restrictions of the natural projections $\tr_\pdn: \ov\mL^\prime \longrightarrow \ov\mL$ and $\tr_\orr: \ov\mL^\prime \longrightarrow \ov\mL^\pp$ induced by the truncation functors $\tr_\pdn$ and $\tr_\orr$, respectively, and they become identity maps.
As $(\ov\mL^\prime)_{\ovla^{(p, r|r, n)}}^\sing=(\mL^\pp)^{\sing}_{\wt\mu^\rr}$ and $(\ov\mL^\prime)_{\ov\xi}^{\sig\mm\sing}=\mL^\sing_{\xi}$,
the composition
$$
\psi:=\varphi \circ \vartheta
$$
is a map from $(\mL^\pp)^{\theta\mm\sing}_{\ov\mu^\rr}$ to $\mL^\sing_{\xi}$.

Recall that $c(z)=\sum_{i=1}^\ell d_i (z-z_i)^{-1}$.
The following theorem not only provides a method of constructing an eigenbasis $\Gamma$ for $(\cA_\pdn^\z)_{\mL^\sing_{\xi}}$ from any eigenbasis $\Gamma^\pp$ for $(\cA_\orr^\z)_{ (\mL^\pp)^{\theta\mm\sing}_{\ov\mu^\rr}}$ but also relates the eigenvalues of $\Gamma^\pp$ to those of $\Gamma$.

\begin{thm} \label{gls-gls-eigen}
\begin{enumerate}[\normalfont(i)]

\item Let $v$ and $w$ be eigenvectors of $(\cA_\orr^\z)_{ (\mL^\pp)^{\theta\mm\sing}_{\ov\mu^\rr}}$ and $(\cA_\pdn^\z)_{\mL^\sing_{\xi}}$, respectively, and let $\cD^\pp_v$ and $\cD_w$ be the pseudo-differential operators associated to $\big(\Ber \big(\cL_\orr^\z \big), v \big)$ and $\big(\Ber \big(\cL_\pdn^\z \big), w \big)$, respectively. 
Then $\psi (v)$ and $\psi^{-1}(w)$ are eigenvectors of $(\cA_\pdn^\z)_{\mL^\sing_{\xi}}$ and $(\cA_\orr^\z)_{ (\mL^\pp)^{\theta\mm\sing}_{\ov\mu^\rr}}$, respectively, and
\begin{equation} \label{eigen-1}
\cD_{\psi(v)} \big( \psi (v) \big) =  \big(\pz+c(z) \big)^{p-q+r}  \cD^\pp_v  \pz^{m-n-r}  \big(\psi (v) \big)
\end{equation}
and
\begin{equation} \label{eigen-2}
\cD_{w} \big( \psi^{-1}(w) \big) =  \big(\pz+c(z) \big)^{p-q+r}  \cD^\pp_{ \psi^{-1}(w)}  \pz^{m-n-r}  \big( \psi^{-1}(w) \big)
\end{equation}

\item The map $\Gamma^\pp  \mapsto \psi(\Gamma^\pp)$, where $\Gamma^\pp$ is an eigenbasis for $(\cA_\orr^\z)_{ (\mL^\pp)^{\theta\mm\sing}_{\ov\mu^\rr}}$, is a one-to-one correspondence between the set of eigenbases for $(\cA_\orr^\z)_{ (\mL^\pp)^{\theta\mm\sing}_{\ov\mu^\rr}}$ and the set of eigenbases for $(\cA_\pdn^\z)_{\mL^\sing_{\xi}}$.

\end{enumerate}
\end{thm}

\begin{proof}
We prove (i), and (ii) is immediate from (i).
We have seen that $(\mL^\pp)_{\wt\mu^\rr}^\sing=(\ov\mL^\pp)_{\ovla^{(0, r|r, 0)}}^\sing=(\ov\mL^\prime)_{\ovla^{(p, r|r, n)}}^\sing$ and $\mL^\sing_{\xi}=\ov\mL_{\ov\xi}^\sing=(\ov\mL^\prime)_{\ov\xi}^{\sig\mm\sing}$.
Let $u=\vartheta(v)$.
Since $\vartheta$ is an $\cA_\orr^\z$-isomorphism, we have
$$
\Ber \big(\cL^\z_\orr \big) u=\cD^\pp_v u.
$$
Using \propref{trunc}, we find that
\begin{eqnarray*}
\Ber \big(\ocL_\pdn^\z \big) \big( \varphi (u) \big)
\ns \nns &=& \nns \big(\pz+c(z) \big)^{-q+r} \, \Ber(\ocL^\z_{(p, r|r,n)}) \, \pz^{m-r} \big( A u \big) \\
\nns &=& \nns A \,  \big(\pz+c(z) \big)^{-q+r} \, \Ber(\ocL^\z_{(p, r|r,n)}) \, \pz^{m-r} u \\
\nns &=& \nns A \,  \big(\pz+c(z) \big)^{p-q+r} \, \Ber(\ocL^\z_{(0, r|r,0)}) \, \pz^{m-n-r} u \\
\nns &=& \nns A \, \big(\pz+c(z) \big)^{p-q+r}  \cD^\pp_v  \pz^{m-n-r} u \\
\nns &=& \nns  \big(\pz+c(z) \big)^{p-q+r}  \cD^\pp_v  \pz^{m-n-r}  \big(\varphi (u) \big).
\end{eqnarray*}
Thus, $\varphi(u)$ is an eigenvector of $(\cA_\pdn^\z)_{\mL^\sing_{\xi}}$, and \eqref{eigen-1} follows.
The remaining parts can be proved similarly.
\end{proof}

Choose a positive integer $r' \ge r$ such that $\ell(\mu) \le r'$.
By our choice of $r'$, the weights $\ov\mu^{r'|r}$ and $\ov\mu_{i}^{r'|r}$ can be regarded as elements of ${\mf X}_{r'|0}^+$ (see \eqref{weight-neg}),
i.e., $\ov\mu^{r'|r}=\ov\mu^{r'|0}$ and $\ov\mu_{i}^{r'|r}=\ov\mu_{i}^{r'|0}$ for all $i=1, \ldots, \ell$.
Let
$$
\mV= L(\mg_{(0,0|r',0)},\ov\mu_1^{r'|0}) \otimes \cdots \otimes L(\mg_{(0,0|r',0)},\ov\mu_\ell^{r'|0})
$$
and
$$
\mV^\prime=L(\mg_{(0,r|r',0)},\wt\mu_1^{r'|r})  \otimes \cdots \otimes L(\mg_{(0,r|r',0)},\wt\mu_\ell^{r'|r}).
$$
By \cite[Theorem 2.55]{CW},
$$
\mV^\prime=L^{\pi}(\mg_{(0,r|r',0)},\ov\mu_1^{r'|r})  \otimes \cdots \otimes L^{\pi}(\mg_{(0,r|r',0)},\ov\mu_\ell^{r'|r}),
$$
where $\pi$ is the permutation of the set $\{-r+\hf, \ldots, -\hf, 1, \ldots, r'\}$ defined by sending $1, \ldots, r', -r+\hf, \ldots, -\hf$ to $-r+\hf, \ldots, -\hf, 1, \ldots, r'$, respectively.
Note that $\otr_{r'|0}(\mV^\prime)= \mV$ and $\otr_{r|r}(\mV^\prime)=\mL^\pp$, where $\otr_{r'|0}: \cC_{r'|r} \longrightarrow \cC_{r'|0}$ and $\otr_{r|r}:\cC_{r'|r} \longrightarrow \cC_{r|r}$ are the truncation functors defined in \cite[Section 3.3]{ChL25}.
By \cite[Proposition 3.14]{ChL25}, there exists $Y \in U(\mg_{(0,r|r',0)})$ such that the map $\varsigma: (\mV^\prime)^{\pi\mm\sing}_{\ov\mu^{r'|r}} \longrightarrow (\mV^\prime)^{\theta^\prime\mm\sing}_{\ov\mu^{r|r}}$, defined by $\varsigma (v) =  Y v$, is an $\cA_{(0,r|r',0)}$-isomorphism.
Here $\theta^\prime$ is the permutation of the set $\{-r+\hf, \ldots, -\hf, 1, \ldots, r'\}$ defined by sending $1, \ldots, r, -r+\hf, \ldots, -\hf, r+1, \ldots, r'$ to $-r+\hf, \ldots, -\hf, 1, \ldots, r'$, respectively.
We have isomorphisms of vector spaces as follows:
$$
\begin{tikzcd}
(\mV^\prime)^{\theta^\prime\mm\sing}_{\ov\mu^{r|r}}\ \arrow[d,swap, "\otr_{r|r}"]
&(\mV^\prime)^{\pi\mm\sing}_{\ov\mu^{r'|r}}=(\mV^\prime)^{\pi\mm\sing}_{\ov\mu^{r'|0}}\arrow[l,swap," \varsigma "]  \arrow[d,"\otr_{r'|0}"] \qquad \qquad    &\\
(\mL^\pp)^{\theta\mm\sing}_{\ov\mu^{r|r}} &\qquad\mV^\sing_{\ov\mu^{r'|0}}&
\end{tikzcd}
$$
Here the vertical maps are restrictions of the natural projections $\otr_{r|r}: \mV^\prime \longrightarrow \mL^\pp$ and $\otr_{r'|0}: \mV^\prime \longrightarrow \mV$ induced by the truncation functors $\otr_{r|r}$ and $\otr_{r'|0}$, respectively, and they become identity maps.
As $(\mV^\prime)^{\theta^\prime\mm\sing}_{\ov\mu^{r|r}}=(\mL^\pp)^{\theta\mm\sing}_{\ov\mu^{r|r}}$ and $(\mV^\prime)^{\pi\mm\sing}_{\ov\mu^{r'|r}}=\mV^\sing_{\ov\mu^{r'|0}}$,
we see that $\varsigma$ is a map from $\mV^\sing_{\ov\mu^{r'|0}}$ to $(\mL^\pp)^{\theta\mm\sing}_{\ov\mu^{r|r}}$.

From now on, let $\Psi=\psi \circ \varsigma$, which maps $\mV^\sing_{\ov\mu^{r'|0}}$ to $\mL^\sing_{\xi}$.
The following theorem enables us to obtain an eigenbasis $\Gamma$ for $(\cA_\pdn^\z)_{\mL^\sing_{\xi}}$ from any eigenbasis $\Upsilon$ for $(\Az_{r'})_{\mV^\sing_{\ov\mu^{r'|0}}}$.
It also relates the eigenvalues of $\Upsilon$ to those of $\Gamma$.

\begin{thm} \label{gl-gls-eigen}
\begin{enumerate}[\normalfont(i)]

\item Let $v$ and $w$ be eigenvectors of $(\Az_{r'})_{\mV^\sing_{\ov\mu^{r'|0}}}$ and $(\cA_\pdn^\z)_{\mL^\sing_{\xi}}$, respectively, and let ${\rm D}_v$ and $\cD_w$ be the pseudo-differential operators associated to $\big(\cdet \big(\cL^\z_{(0,0|r',0)} \big), v \big)$ and $\big(\Ber \big(\cLs^\z \big), w \big)$, respectively. 
Then $\Psi (v)$ and $\Psi^{-1}(w)$ are eigenvectors of $(\cA_\pdn^\z)_{\mL^\sing_{\xi}}$ and $(\Az_{r'})_{\mV^\sing_{\ov\mu^{r'|0}}}$, respectively, and
$$
\cD_{\Psi(v)} (\Psi(v))
 = \big(\pz+c(z) \big)^{p-q+r}  {\rm D}_v \pz^{m-n-r-r'} (\Psi(v))
 $$
 and
 $$
\cD_w (\Psi^{-1}(w))
 = \big(\pz+c(z) \big)^{p-q+r}  {\rm D}_{\Psi^{-1}(w)} \pz^{m-n-r-r'} (\Psi^{-1}(w)).
$$

\item The map $\Upsilon  \mapsto \Psi(\Upsilon)$, where $\Upsilon$ is an eigenbasis for $(\Az_{r'})_{\mV^\sing_{\ov\mu^{r'|0}}}$, is a one-to-one correspondence between the set of eigenbases for $(\Az_{r'})_{\mV^\sing_{\ov\mu^{r'|0}}}$ and the set of eigenbases for $(\cA_\pdn^\z)_{\mL^\sing_{\xi}}$.

\end{enumerate}
\end{thm}

\begin{proof}
This follows from \thmref{gls-gls-eigen} and \cite[Proposition 4.9]{ChL25}.
\end{proof}

The following is a special case of \cite[Theorem 5.1]{ChL25}.

\begin{prop} \label{Bethe-gl-eigen}
 If the Bethe ansatz equations hold for $(\Az_{r'})_{\mV^\sing_{\ov\mu^{r'|0}}}$, and $v$ is a Bethe vector in $\mV^\sing_{\ov\mu^{r'|0}}$, then $w:= \varsigma (v)$ is an eigenvector of $(\cA_\orr)_{(\mL^\pp)^{\theta\mm\sing}_{\ov\mu^{r|r}}}$.
Moreover,
$$
\Ber \big(\cL^\z_\orr \big)w
=  \bD_{r'}  \pz^{-r'}w .
$$
\end{prop}

\begin{thm} \label{super-Bethe}
If the Bethe ansatz equations hold for $(\Az_{r'})_{\mV^\sing_{\ov\mu^{r'|0}}}$, and $v$ is a Bethe vector in $\mV^\sing_{\ov\mu^{r'|0}}$, then $w:=\Psi (v)$ is an eigenvector of $(\cA_\pdn^\z)_{\mL^\sing_{\xi}}$.
Moreover,
$$
\Ber \big(\cLs^\z \big) w
 = \big(\pz+c(z) \big)^{p-q+r}  \bD_{r'}  \pz^{m-n-r-r'} w.
$$

\end{thm}

\begin{proof}
This follows from \thmref{gls-gls-eigen} and \propref{Bethe-gl-eigen}.
\end{proof}

We call \thmref{super-Bethe} the \emph{super Bethe ansatz} for $\cA_\pdn^\z$ on $\mL^\sing_\xi$.
If the completeness of the Bethe ansatz is true for $(\Az_{r'})_{\mV^\sing_{\ov\mu^{r'|0}}}$, then $\Psi (\fB)$ is an eigenbasis for $(\cA_\pdn^\z)_{\mL^\sing_{\xi}}$ for a generic $\z$, where $\fB$ is the set consisting of Bethe vectors in $\mV^\sing_{\ov\mu^{r'|0}}$.

The following example is similar to \cite[Example 5.3]{ChL25} with a slight modification.

\begin{ex} \label{ex-Bethe}
Let $\z=(z_1, z_2) \in {\bd X}_{\! 2}$, and let $\mL^\pp=\Lt(\mg_{(0,2|2,0)}, \ov\mu_1^{2|2})  \otimes \Lt(\mg_{(0,2|2,0)}, \ov\mu_2^{2|2})$,
where $\mu_1:=(3, 2, 1, 1)$ and $\mu_2:=(1)$ are $(2|2)$-hook partitions, and $\theta$ is the permutation of the set $\{-\thf, -\hf, 1, 2\}$ defined by sending $1, 2, -\thf, -\hf$ to $-\thf, -\hf, 1, 2$, respectively.
We consider the Gaudin algebra $\big(\cA_{(0,2|2,0)} ^\z\big)_{(\mL^\pp)^{\theta\mm\sing}_{\ov\mu^{2|2}}}$, where $\mu:=(3, 2, 2, 1) \in \cP_{2|2}$.
We find that
$$
\ov\mu_1^{2|2}=3\ep_1+2\ep_2+ 2\ep_{-\thf}, \qquad \ov\mu_2^{2|2}=\ep_1,  \qquad  \ov\mu^{2|2}=3\ep_1+2\ep_2+2 \ep_{-\thf} +\ep_{-\hf}.
$$
Also, $\ov\mu_1^{4|2}=3 \ep_1+2 \ep_2 + \ep_3 + \ep_4=\ov\mu_1^{4|0}$, $\ov\mu_2^{4|2}=\ep_1=\ov\mu_2^{4|0}$, and $\ov\mu^{4|2}=3 \ep_1+2 \ep_2 + 2 \ep_3 + \ep_4=\ov\mu^{4|0}$.

Let $\mV=L(\mg_{(0,0|4,0)}, \ov\mu_1^{4|0}) \otimes L(\mg_{(0,0|4,0)}, \ov\mu_2^{4|0})$.
Suppose \eqref{BAn} are satisfied, i.e.,
$w_1=\frac{1}{4}(z_1+3z_2)$ and $w_2=\frac{1}{8}(5z_1+3z_2)$ (see \emph{loc. cit.}), and $\big|w_1^1, w_2^2 \big\rangle$ is nonzero.
Then $\big|w_1^1, w_2^2 \big\rangle$ is an eigenvector of $\big(\cA_4^\z\big)_{\mV^\sing_{\ov\mu^{4|0}}}$.
Choose $ \varsigma: \mV^\sing_{\ov\mu^{4|0}} \longrightarrow (\mL^\pp)^{\theta\mm\sing}_{\ov\mu^{2|2}}$ such that
\begin{equation} \label{w}
w:=\varsigma \big( \big|w_1^1, w_2^2 \big\rangle \big)= E_{-\hf, 3} E_{-\thf, 3}  E_{-\thf, 4}   \big|w_1^1, w_2^2 \big\rangle.
\end{equation}
By \propref{Bethe-gl-eigen}, $w$ is an eigenvector of $\big(\cA_{(0,2|2,0)} ^\z\big)_{(\mL^\pp)^{\theta\mm\sing}_{\ov\mu^{2|2}}}$, and the equation
$$
\Ber \big(\ocL_{(0,2|2,0)} ^\z \big) w =\bD_4 \pz^{-4} w
$$
encodes the eigenvalues of $w$.
\qed
\end{ex}

The following provides an example of the Gaudin algebra for a general linear Lie algebra acting on a tensor product of infinite-dimensional unitarizable highest weight modules.

\begin{ex}
Suppose $(p,q,m,n)=(4,0,1,0)$.
Let  $\Is=\Is^{-}_{4|0} \cup \Is^{+}_{1|0}$.
We have $\mg_{(4,0|1,0)} \cong \gl_5$.
Let $L_1=L(\mg_{(4,0|1,0)},\xi_1)$ and $L_2= L(\mg_{(4,0|1,0)},\xi_2)$, where $\xi_1:=-5\ep_{-4}-6\ep_{-3}-6\ep_{-2}-7\ep_{-1}+\ep_1 \in {\mf X}^5_{(4,0|1,0)}$ and $\xi_2:=-2\ep_{-4}-2\ep_{-3}-3\ep_{-2}-4\ep_{-1}  \in {\mf X}^2_{(4,0|1,0)}$.
The spaces $L_1$ and $L_2$ are infinite-dimensional unitarizable highest weight $\mg_{(4,0|1,0)}$-modules.

Let $\z=(z_1, z_2) \in {\bd X}_{\! 2}$, and let $\mL=L_1 \otimes L_2$.
We consider the Gaudin algebra $\big(\cA_{(4,0|1,0)}^\z\big)_{\mL^\sing_{\xi}}$, where $\xi:=-8\ep_{-4}-9\ep_{-3}-9\ep_{-2}-9\ep_{-1}+\ep_1 \in {\mf X}^7_{(4,0|1,0)}$.
We calculate that
\begin{eqnarray*}
\ov\xi_1\hspace{-2mm} &=& \hspace{-2mm} -\ep_{-3}-\ep_{-2}-2\ep_{-1}+\ep_1+5 \La_0 \in \ov{\mf X}^5_{(4,0|1,0)} , \\
\ov\xi_2 \hspace{-2mm} &=& \hspace{-2mm} -\ep_{-2}-2\ep_{-1}+2 \La_0 \in \ov{\mf X}^2_{(4,0|1,0)}, \\
\ov\xi \hspace{-2mm} &=& \hspace{-2mm} -\ep_{-4}-2\ep_{-3}-2\ep_{-2}-2\ep_{-1}+\ep_1 +7 \La_0 \in \ov{\mf X}^7_{(4,0|1,0)}.
\end{eqnarray*}
Write $\ov\xi_i=\ovla_{(i)}^{(4,0|1,0)}$, for $i=1, 2$, and $\ov\xi=\ovla^{(4,0|1,0)} $, where
$\ovla_{(1)}=\big((1), (2, 1,1); 5 \big)\in \cP^5_{(4,0|1,0)}$,
$\ovla_{(2)}=\big(\emptyset, (2, 1); 2 \big)\in \cP^2_{(4,0|1,0)}$, and
$\ovla=\big((1), (2, 2, 2,1); 7 \big)\in \cP^7_{(4,0|1,0)}$.
We see that
\begin{eqnarray*}
\ovla_{(1)}^{(4,2|2,0)} \hspace{-2mm} &=& \hspace{-2mm} -\ep_{-\thf}-3\ep_{-\hf}+\ep_1+5 \La_0 \in \ov{\mf X}^5_{(4,2|2,0)}, \\
\ovla_{(2)}^{(4,2|2,0)} \hspace{-2mm} &=& \hspace{-2mm} -\ep_{-\thf}-2\ep_{-\hf}+2 \La_0 \in \ov{\mf X}^2_{(4,2|2,0)}, \\
\ovla^{(4,2|2,0)} \hspace{-2mm} &=& \hspace{-2mm} -3\ep_{-\thf}-4\ep_{-\hf}+\ep_1+7 \La_0 \in \ov{\mf X}^7_{(4,2|2,0)}.
\end{eqnarray*}
Let $\sig=\sig_{(0,1)} \in \fS_\Is$ (see \exref{ex}).
Then $(\ovla_{(i)}^{(4,2|2,0)})^\sig=\ov\xi_i$, for $i=1, 2$, and $(\ovla^{(4,2|2,0)})^\sig=\ov\xi$.
As in \eqref{Arr-iso}, we may choose $\varphi: (\ov\mL^\prime)_{\ovla^{(4,2|2,0)}}^\sing \longrightarrow (\ov\mL^\prime)_{\ov\xi}^{\sig\mm\sing}$ such that
\begin{eqnarray*}
v&:=& \big( \varphi \circ \vartheta \big) (w) \\
&\hspace{1.5mm}=& \ov{E}_{-\hf,-4} \ov{E}_{-\hf,-3}  \ov{E}_{-\thf,-3} \ov{E}_{-\hf,-2}  \ov{E}_{-\thf,-2}  \ov{E}_{-\hf,-1}  \ov{E}_{-\thf,-1} \vartheta (w),
\end{eqnarray*}
where $w$ is defined in \eqref{w}.

Let $\mu_1$, $\mu_2$ and $\mu$ be the $(2|2)$-hook partitions defined in \exref{ex-Bethe}. We have
$$
\wt\mu_1^{2|2}= 4 \ep_{-\thf}+2 \ep_{-\hf}+\ep_1, \quad \wt\mu_2^{2|2} = \ep_{-\thf}, \quad \wt\mu^{2|2} = 4\ep_{-\thf}+3\ep_{-\hf}+\ep_1 \in \mh_{(0,2|2,0)}^*.
$$
Now choose $\vartheta$ such that
$$
\vartheta (w)= E_{-\hf, 1} E_{-\thf, 1}  E_{-\hf, 2}  E_{-\thf, 2}   w.
$$
In this way, $v  \in \mL_{\xi}^\sing$ has an expression in terms of $E_{i,j}$'s and $\big|w_1^1, w_2^2 \big\rangle$.
By \thmref{super-Bethe}, $v$ is an eigenvector of $\big(\cA_{(4,0|1,0)}^\z\big)_{\mL^\sing_{\xi}}$, and the equation
$$
\cdet \big(\cL_{(4,0|1,0)}^\z \big) v
 = \big(\pz+c(z) \big)^{6}  \bD_4  \pz^{-5} v
$$
encodes the eigenvalues of $v$, where $c(z):=5 (z-z_1)^{-1}+2 (z-z_2)^{-1}$.
\qed
\end{ex}

\vspace{0.5cm}

\noindent{\bf Acknowledgments.}
The first author was partially supported by NSFC (Grant No. 12161090).
The second author was partially supported by NSTC grant 115-2115-M-006-012-MY2.
The third author was partially supported by NSTC grant 112-2115-M-006-015-MY2, and he thanks the School of Mathematics and Statistics at Yunnan University in Kunming for hospitality and support.

\end{document}